\documentclass{amsart}

\usepackage{amssymb}

\newtheorem{theorem}{Theorem}[section]
\newtheorem{lemma}[theorem]{Lemma}
\newtheorem{proposition}[theorem]{Proposition}
\newtheorem{corollary}[theorem]{Corollary}

\theoremstyle{definition}

\theoremstyle{remark}
\newtheorem{remark}[theorem]{Remark}

\numberwithin{equation}{section}

\begin{document}

\title{On the equality of operator valued weights}

\author{L\'aszl\'o Zsid\'o}
\address{Department of Mathematics, University of Rome "Tor Vergata",
Via della Ricerca Scientifica, 00133 Roma, Italia}
\email{zsido@mat.uniroma2.it}
\thanks{}



\date{June 27, 2021.}
\thanks{The author was supported by GNAMPA-INDAM and ERC Advanced Grant 669240 QUEST}
\dedicatory{This paper is dedicated to the memory of Uffe Haagerup.}


\begin{abstract}
G. K. Pedersen and M. Takesaki have proved in 1973 that if $\varphi$ is a faithful, semi-finite,
normal weight on a von Neumann algebra $M\;\!$, and $\psi$ is a $\sigma^{\varphi}$-invariant,
semi-finite, normal weight on $M\;\!$, equal to $\varphi$ on the positive part of a weak${}^*$-dense
$\sigma^{\varphi}$-invariant $*$-subalgebra of $\mathfrak{M}_{\varphi}\;\!$, then $\psi =\varphi\;\!$.

In 1978 L. Zsid\'o extended the above result by proving: if $\varphi$ is as above,
$a\geq 0$ belongs to the centralizer $M^{\varphi}$ of $\varphi\;\!$,
and $\psi$ is a $\sigma^{\varphi}$-invariant, semi-finite, normal weight
on $M\;\!$, equal to $\varphi_a:=\varphi (a^{1/2}\;\!\cdot\;\! a^{1/2})$ on the positive part of a
weak${}^*$-dense $\sigma^{\varphi}$-invariant $*$-subalgebra of $\mathfrak{M}_{\varphi}\;\!$,
then $\psi =\varphi_a\;\!$.

Here we will further extend this latter result, proving criteria for both the inequality
$\psi \leq\varphi_a$ and the equality $\psi =\varphi_a\;\!$. Particular attention is
accorded to criteria with no commutation assumption between $\varphi$ and $\psi\;\!$,
in order to be used to prove inequality and equality criteria for operator valued weights.

Concerning operator valued weights, it is proved that if $E_1\;\! ,E_2$ are semi-finite, normal
operator valued weights from a von Neumann algebra $M$ to a von Neumann subalgebra
$N\ni 1_M$ and they are equal on $\mathfrak{M}_{E_1}\;\!$, then $E_2\leq E_1\;\!$.
Moreover, it is shown that this happens if and only if for any (or, if $E_1\;\! ,E_2$ have
equal supports, for some) faithful, semi-finite, normal weight $\theta$ on $N$
the weights $\theta\circ E_2\;\! ,\theta\circ E_1$ coincide on $\mathfrak{M}_{\theta\circ E_1}\;\!$.
\end{abstract}

\maketitle

\section*{Introduction}

For the equality of two normal positive forms on a $W^*$-algebra it is enough that
they coincide on a weak${}^*$-dense subset. For unbounded normal weights
equality can follow from equality on a weak${}^*$-dense subset only under additional
conditions on the weights and/or on the subset.

A first criterion of this type, which served as a model for subsequent criteria,
is \cite{PT}, Proposition 5.9 : Let $\varphi\;\! ,\psi$ be semi-finite, normal weights
on a $W^*$-algebra $M$, $\varphi$ faithful and $\psi$ $\sigma^{\varphi}$-invariant.
If $\psi (x^*x)=\varphi (x^*x)$ for all $x$ in a weak${}^*$-dense,
$\sigma^{\varphi}$-invariant ${}^*$-subalgebra of $\mathfrak{M}_{\varphi}\;\!$,
then $\psi =\varphi\;\!$.
This citerion was further extended in \cite{Z2}, Theorem 2.3 as follows :
Let $\varphi\;\! ,\psi$ be as above, and $a$ a positive element of the centralizer of
$\varphi\;\!$. If $\psi (x^*x)=\varphi (a^{1/2}x^*x\;\! a^{1/2})$ for $x$ in a
weak${}^*$-dense, $\sigma^{\varphi}$-invariant ${}^*$-subalgebra of
$\mathfrak{M}_{\varphi}\;\!$, then $\psi =\varphi (a^{1/2}\;\!\cdot\;\! a^{1/2})\;\!$.

In both of the above results the $\sigma^{\varphi}$-invariance of $\psi$ is assumed.
Therefore they are not useful to prove equality criteria for operator valued weights.
Indeed, if $M$ is a $W^*$-algebra and $1_M\in N\subset M$ a $W^*$-subalgebra, then,
according to \cite{H1}, Lemma 4.8, the equality of two faithful, semi-finite, normal
operator valued weights $E_1\;\! ,E_2$ from $M$ to $N$ would be implied by the
equality of $\theta\circ E_1=\theta\circ E_2$ for some faithful, semi-finite, normal
weight $\theta$ on $N$. However, there is no known way to derive the
$\sigma^{\theta\circ E_1}$-invariance of $\theta\circ E_2$ from appropriate
properties of $E_1\;\! ,E_2$ and a suitable choice of $\theta\;\!$.
Consequently, having in view equality criteria for normal operator
valued weights, it is of interest to prove criteria for the equality of normal weights
with no commutation assumption for them.

Similar considerations support the interest to prove inequality criteria for
normal weights without making any commutation assumption.

After a preliminary first section concerning notations and used results, the
second section is dedicated to the so-called {\it regularizing nets}, a useful tool
in the modular theory of von Neumann algebras. It was used in the proof
of \cite{PT}, Lemma 5.2 and of \cite{Z2}, Lemma 2.1, on which \cite{PT}, Proposition 5.9
resp.  \cite{Z2}, Theorem 2.3 are based. It will be an essential ingredient also of the
proof of the main result of Section 3.

We will call regularizing net for a faithful, semi-finite, normal weight $\varphi$ on a

\noindent $W^*$-algebra $M$ (cf. \cite{Z2}, \S 1) any net $(u_{\iota})_{\iota}$ in the
Tomita algebra of $\varphi$ such that
\begin{itemize}
\item [(i)] $\sup\limits_{\substack{ z\in K \\ \iota }}\|\sigma^{\varphi}_z(u_{\iota})\|
<+\infty$ and $\sup\limits_{\substack{ z\in K \\ \iota }}\|\sigma^{\varphi}_z(u_{\iota})_{\varphi}\|
<+\infty$ for each compact $K\subset\mathbb{C}\;\!$;
\item[(ii)] $\sigma^{\varphi}_z(u_\iota) \xrightarrow{\;\; \iota\;\,}1_M$ in the $s^*$-topology
for all $z\in\mathbb{C}\;\!$.
\end{itemize}
They can be derived from a bounded net $(x_\iota )_{\iota}$ in $\mathfrak{N}_{\varphi}^*
\cap \mathfrak{N}_{\varphi}$ such that $x_{\iota} \xrightarrow{\;\; \iota\;\,}1_M$ in the $s^*$-topology
by setting
\begin{equation}\label{mollify-pre}
u_{\iota} =\frac 1{\;\!\sqrt{\pi}\;\!}\! \int\limits_{-\infty}^{+\infty}\!\!
e^{-t^2}\! {\sigma}^{\varphi}_t(x_{\iota})\;\!{\rm d}t\;\! . \tag{\ref{mollify}}
\end{equation}
However the verification that the such obtained net $(u_{\iota})_{\iota}$ satisfies (ii), done
in the proof of \cite{PT}, Lemma 5.2, works only if the net $(x_\iota )_{\iota}$ is increasing,
while in the proof of \cite{Z2}, Lemma 2.1 no verification is given.
If $(x_\iota )_{\iota}$ would be a sequence, we could use the dominated convergence theorem
of Lebesgue, but the general case needs a complete treatment. We will prove that, even if we
assume only $x_{\iota}\in \mathfrak{N}_{\varphi}\;\!$, formula (\ref{mollify-pre}) yields always a
regularizing net for $\varphi$ (Theorem \ref{reg.net}).

Section 3 is dedicated to criteria for inequalities and equality between normal weights.
The main result is an extension of \cite{Z2}, Lemma 2.1 : If $\varphi\;\! ,\psi$ are semi-finite,
normal weights on a $W^*$-algebra $M$ with $\varphi$ assumed faithful, $a\geq 0$ is an element
of the centralizer $M^{\varphi}$ of $\varphi\;\!$, and, setting $\varphi_a:=
\varphi (a^{1/2}\;\!\cdot\;\! a^{1/2})\;\!$, $\psi (x^*x)=\varphi_a(x^*x)$ holds for every $x$ in a
weak${}^*$-dense, $\sigma^{\varphi}$-invariant ${}^*$-subalgebra of
$\mathfrak{M}_{\varphi_a}\;\!$, then $\psi \leq \varphi_a$ (Theorem \ref{ineq}).
If we assume additionally that $\psi$ is $\sigma^{\varphi}$-invariant, then the
equality $\psi =\varphi_a$ follows (Theorem \ref{eq-1}).
But the above inequality criterion implies also the following equality criterion :
If $\varphi\;\! ,\psi$ are faithful, semi-finite, normal weights on a  $W^*$-algebra $M$, $a ,b\geq 0$ 
are elements of $M^{\varphi}$ resp. $M^{\psi}$, and $\psi_b (x^*x)={\varphi}_a(x^*x)$

\noindent holds for all $x$ in a weak${}^*$-dense ${}^*$-subalgebra of
$\mathfrak{M}_{\varphi_a}\cap\mathfrak{M}_{\psi_b}\;\!$, which is both $\sigma^{\varphi}$- and
$\sigma^{\psi}$-invariant, then $\psi_b = \varphi_a$ (Theorem \ref{eq-4}).
We underline that, neither in the inequality criterion nor in the above result $\sigma^{\varphi}$- or
$\sigma^{\psi}$-invariance of the other weight is required.

In the last Section 4 we discuss inequality and equality criteria for semi-finite, normal operator valued
weights $E_1\;\! ,E_2$ from a $W^*$-algebra $M$ to a $W^*$-subalgebra $1_M\in N\subset M$.
Reduction to the case of scalar weights, considered in Section 3, is used.

The main result of the section is Theorem \ref{general-ineq}, which consists of two parts.
In part (i) it is proved that if $\mathfrak{M}_{E_1}\!\subset \mathfrak{M}_{E_2}$ and $E_1\;\! ,E_2$
coincide on $\mathfrak{M}_{E_1}\;\!$, then $E_2\leq E_1\;\!$, while in part (ii) it is shown that
$\mathfrak{M}_{E_1}\!\subset \mathfrak{M}_{E_2}$ and $E_1=E_2$ on $\mathfrak{M}_{E_1}$
happens if and only if $\mathfrak{M}_{\theta\circ E_1}\!\subset \mathfrak{M}_{\theta\circ E_2}$
and  $\theta\circ E_1=\theta\circ E_2$ on $\mathfrak{M}_{\theta\circ E_1}$ for all faithful, semi-finite,
normal weights $\theta$ on $N$. We notice that the difficulty in proving (ii) arises from the
fact that there is no generic relation between $\mathfrak{M}_{E_1}$ and $\mathfrak{M}_{\theta\circ E_1}\;\!$.

If $E_1$ and $E_2$ have equal supports, for $\mathfrak{M}_{E_1}\!\subset \mathfrak{M}_{E_2}$
and $E_1=E_2$ on $\mathfrak{M}_{E_1}$ another characterization, based on the Connes cocycle,
is provided (Theorem \ref{compl-ineq}). Most likely the support condition is here redundant.

\section{Preliminaries}
We will use the terminology of \cite{SZ1} and \cite{St}. In particular,
\begin{itemize}
\item $(\;\!\cdot\;\! |\;\!\cdot\;\! )$ will denote the inner product of a Hilbert space and it will be
assumed linear in the first variable and antilinear in the second variable;
\item $B(H)$ will denote the algebra of all bounded linear operators on the Hilbert space $H$,
with the identity operator denoted by $1_H\;\!$;
\item for $\xi$ in a Hilbert space $H$,
$\omega_{\xi}$ will denote the positive vector functional

\noindent $B(H)\ni x\longmapsto (x\;\! \xi | \xi )\;\!$;
\item the unit of a $W^*$ algebra $M$ will be denoted by $1_M\;\!$;
\item the $w$-topology on a $W^*$-algebra $M$ is the weak${}^*$topology, the locally convex
topology defined by the semi-norms
\smallskip

\centerline{$\qquad\qquad M\ni x\longmapsto |\varphi (x)|\;\!$, $\varphi$ a normal positive form on $M$;}
\smallskip
\item the $s^*$-topology on a $W^*$-algebra $M$ is the locally comvex topologies defined by the semi-norms
\smallskip

\centerline{$\qquad\qquad p_{\varphi} : M\ni x\longmapsto \varphi (x^*x)^{1/2}+ \varphi (x\;\! x^*)^{1/2}\;\!$,
$\varphi$ a normal positive form on $M$;}
\smallskip

\item $s(A)$ denotes the support projection of a self-adjoint operator $A$ affiliated to a
$W^*$-algebra $M$ (does not matter in which spatial representation);
\item $s(\varphi )$ denotes the support projection of a normal positive form or a normal weight $\varphi$
on a $W^*$-algebra $M$;
\item $\overline{M}\;\!^+$ denotes the extended positive part of a $W^*$-algebra $M$.
\end{itemize}

As usual, for a weight $\varphi$ on a $W^*$-algebra $M$ we use the notations
\smallskip

\noindent\hspace{3.9 cm}$\mathfrak{N}_{\varphi}=\{ x\in M ; \varphi (x^*x)<+\infty\}\;\! ,$
\smallskip

\noindent\hspace{3.94 cm}$\mathfrak{A}_{\varphi}=\mathfrak{N}_{\varphi}\cap
\mathfrak{N}_{\varphi}{\!\!\!}^*\;\! ,$
\smallskip

\noindent\hspace{3.82 cm}$\mathfrak{M}_{\varphi}=\text{linear span of }\mathfrak{N}_{\varphi}{\!\!\!}^*
\mathfrak{N}_{\varphi}\;\! .$
\smallskip

\noindent Then
\smallskip

\centerline{$\mathfrak{M}_{\varphi}\cap M^+=\{ a\in M^+ ;\varphi (a)<+\infty\}\;\! .$}
\smallskip

We notice that, for $\varphi$ a weight on a $W^*$-algebra $M$ and $e\in M$ a projection,
\begin{equation}\label{reduced}
\varphi (1_M-e)=0\;\! \Longrightarrow \varphi (x^*x)=\varphi (e\;\! x^*x\;\! e)\text{ for all }x\in M\;\! .
\end{equation}
Indeed, since $x-x\;\! e=x\;\! (1_M-e)\in \mathfrak{N}_{\varphi}\;\!$, we have $x\in \mathfrak{N}_{\varphi}
\Longleftrightarrow x\;\! e\in \mathfrak{N}_{\varphi}\;\!$. Thus, if $x\notin  \mathfrak{N}_{\varphi}\;\!$,
then $\varphi (x^*x)= +\infty =\varphi (e\;\! x^*x\;\! e)\;\!$. On the other hand, if $x$ and $x\;\! e$ belong
to $\mathfrak{N}_{\varphi}\;\!$, then the Schwarz inequality entails

\begin{equation*}
\begin{split}
\big|\varphi \big( (1_M-e)\;\! x^*x\;\! e\big) \big|^2\leq\;& \varphi \big( (1_M-e)\;\! x^*x\;\! (1_M-e)\big)
\varphi (e\;\! x^*x\;\! e) \\
\leq\;& \| x^*x\| \varphi (1_M-e) \varphi (e\;\! x^*x\;\! e) =0
\end{split}
\end{equation*}
and, similarly, $\varphi \big( e\;\! x^*x\;\! (1_M-e)\big) =0\;\!$. Consequently,
\begin{equation*}
\begin{split}
&\varphi (x^*x) \\
=\;&\varphi (e\;\! x^*x\;\! e) +\varphi \big( (1_M-e)\;\! x^*x\;\! e\big)
+\varphi \big( e\;\! x^*x\;\! (1_M-e)\big) +\varphi \big( (1_M-e)\;\! x^*x\;\! (1_M-e)\big) \\
=\;& \varphi (e\;\! x^*x\;\! e)\;\! .
\end{split}
\end{equation*}

For $\varphi$ a semi-finite, normal weight on a $W^*$-algebra $M$, then $\pi_{\varphi} : M\longrightarrow
B(H_{\varphi})$ denotes its GNS representation, and $x_{\varphi}$ the canonical image of
$x\in\mathfrak{N}_{\varphi}$ in $H_{\varphi}\;\!$.
If $\varphi$

\noindent is also faithful, then $S_{\varphi}$ stands for the closure of the antilinear operator
\smallskip

\centerline{$H_{\varphi}\supset (\mathfrak{A}_{\varphi})_{\varphi}\ni x_{\varphi}
\longmapsto (x^*)_{\varphi}\in H_{\varphi}\;\! ,$}
\smallskip

\noindent $\Delta_{\varphi}=S_{\varphi}^*S_{\varphi}$ for the modular operator of $\varphi\;\!$,
and $J_{\varphi}$ for its modular conjugation, so that $S_{\varphi}=J_{\varphi}\Delta_{\varphi}^{1/2}\;\!$.
Further, $\big(\sigma^{\varphi}_t\big)_{t\in\mathbb{R}}$ will denote the modular automorphism group
of $\varphi\;\!$:

\centerline{$\pi_{\varphi}\big(\sigma^{\varphi}_t(x)\big)=\Delta_{\varphi}^{it}\pi_{\varphi}(x)\Delta_{\varphi}^{-it},
\qquad t\in\mathbb{R}\;\! , x\in M .$}
\medskip

\noindent An element $x\in M$ will be called $\sigma_{\varphi}$-entire if $\mathbb{R}\ni t\longmapsto
\sigma^{\varphi}_t(x)\in M$ has an entire

\noindent extension, denoted $\mathbb{C}\ni z\longmapsto \sigma^{\varphi}_z(x)\in M$, and we consider
the ${}^*$-subalgebras of $M$
\smallskip

\centerline{$M^{\varphi}_{\infty}=\{ x\in M\;\! ;\;\! x\text{ is }\sigma_{\varphi}\text{-entire}\}\;\! ,$}
\smallskip

\centerline{$M^{\varphi}=\{ x\in M\;\! ;\;\! \sigma^{\varphi}_t(x)=x\text{ for all }t\in\mathbb{C}\}
\text{ (centralizer of }\varphi )\;\! ,$}
\smallskip

\centerline{$\mathfrak{T}_{\varphi}=\{ x\in M^{\varphi}_{\infty}\;\! ;\;\! \sigma^{\varphi}_z(x)\in
\mathfrak{A}_{\varphi}\text{ for all }t\in\mathbb{R}\}
\text{ (maximal Tomita algebra of }\varphi )\;\! .$}
\smallskip

\noindent For $x\in \mathfrak{T}_{\varphi}\;\!$, the vector $x_{\varphi}$ belongs to the domain
of each $\Delta_{\varphi}^{iz}\;\! , z\in\mathbb{C}\;\!$, and
\smallskip

\centerline{$\sigma^{\varphi}_z(x)_{\varphi}=\Delta_{\varphi}^{iz} x_{\varphi}\;\! ,\qquad z\in\mathbb{C}$}
\smallskip

\noindent (see e.g. \cite{SZ1}, Sections 10.20 and 10.21).

If $\varphi$ is a not necessarily faithful semi-finite, normal weight on a $W^*$-algebra $M$, then the modular
authomorphisms are considered acting on the reduced algebra $s(\varphi )Ms(\varphi )$ and
$M^{\varphi}_{\infty}\;\!$, $M^{\varphi}$, $\mathfrak{T}_{\varphi}$ will be ${}^*$-subalgebras
of $s(\varphi )Ms(\varphi )\;\!$.
\smallskip

Let $\Omega$ be a locally compact topological space, $\mu$ a Radon measure on $\Omega\;\!$,
$M$ a
\medskip

\noindent $W^*$-algebra, and $F : \Omega\longrightarrow M$ a $w$-continuous function such that
\medskip

\centerline{$\displaystyle \int\limits_{\Omega}\! \| F(\omega )\|\;\!{\rm d}\mu (\omega )<+\infty\;\! .$}

\noindent Then $\displaystyle M_*\ni\varphi\longmapsto \int\limits_{\Omega}\!\varphi\big( F(\omega )\big)
\;\!{\rm d}\mu (\omega )\in\mathbb{C}$ is a bounded linear functional, hence an element of
$(M_*)^*=M$, denoted by $\displaystyle \int\limits_{\Omega}\! F(\omega )\;\!{\rm d}\mu (\omega )\;\!$.
Therefore
\begin{equation}\label{integration}
\varphi\Big( \int\limits_{\Omega}\! F(\omega )\;\!{\rm d}\mu (\omega )\Big)=
 \int\limits_{\Omega}\!\varphi\big( F(\omega )\big)\;\!{\rm d}\mu (\omega )
\end{equation}
for every $\varphi\in M_*\;\!$.
If $F(\Omega )\subset M^+$ then (\ref{integration}) holds also for any normal weight $\varphi$
on $M$ (see \cite{PT}, Lemma 3.1). Indeed, according to \cite{H}, we have
\medskip

\centerline{$\displaystyle \varphi (a)=\sum\limits_{\iota}\varphi_{\iota}(a)\;\! ,\qquad a\in M^+$}
\smallskip

\noindent for some family $\big(\varphi_{\iota}\big)_{\iota}$ of normal positive forms on $M$ and,
taking into account that the monotone convergence theorem for lower semicontinuous positive
functions and regular Borel measures applies to arbitrary upward directed families (see e.g.
\cite{C},

\noindent Proposition 7.4.4), we deduce:
\begin{equation*} 
\begin{split}
\varphi\Big( \int\limits_{\Omega}\! F(\omega )\;\!{\rm d}\mu (\omega )\Big)=\;&
\sum\limits_{\iota}\varphi_{\iota}\Big( \int\limits_{\Omega}\! F(\omega )\;\!{\rm d}\mu (\omega )\Big)
\overset{(\ref{integration})}{=}
\sum\limits_{\iota} \int\limits_{\Omega}\!\varphi_{\iota}\big( F(\omega )\big)\;\!{\rm d}\mu (\omega ) \\
=\;&\int\limits_{\Omega}\! \sum\limits_{\iota} \varphi_{\iota}\big( F(\omega )\big)\;\!{\rm d}\mu (\omega )
= \int\limits_{\Omega}\!\varphi\big( F(\omega )\big)\;\!{\rm d}\mu (\omega )\;\! .
\end{split}
\end{equation*}
We notice also that
\begin{equation}\label{semi-norm-ineq}
p_{\varphi} \Big( \int\limits_{\Omega}\! F(\omega )\;\!{\rm d}\mu (\omega )\Big) \leq
\int\limits_{\Omega}\! p_{\varphi} \big( F(\omega )\big)\;\!{\rm d}\mu (\omega )\;\! ,
\end{equation}
where $\varphi$ is an arbitrary normal positive form on $M$ and $p_{\varphi}$ denotes the
semi-norm $M\ni x\longmapsto \varphi (x^*x)^{1/2}+ \varphi (x\;\! x^*)^{1/2}\;\!$.
Inequality (\ref{semi-norm-ineq}) is consequence of
\begin{equation}\label{semi-norm-ineq-sep}
p_{\varphi ,j} \Big( \int\limits_{\Omega}\! F(\omega )\;\!{\rm d}\mu (\omega )\Big) \leq
\int\limits_{\Omega}\! p_{\varphi ,j} \big( F(\omega )\big)\;\!{\rm d}\mu (\omega )\;\! ,\qquad
j=1\;\! ,2\;\! ,
\end{equation}
where the semi-norms $p_{\varphi ,1}$ and $p_{\varphi ,2}$ on $M$ are defined by
\smallskip

\centerline{$p_{\varphi ,1}(x)=\varphi (x^*x)^{1/2}\;\! ,\;\! p_{\varphi ,2}(x)=\varphi (x\;\! x^*)^{1/2}\;\!,
\qquad x\in M\;\! .$}
\medskip

\noindent Now (\ref{semi-norm-ineq-sep}) easily follows by using (\ref{integration}) and the relations
\medskip

\centerline{$\displaystyle p_{\varphi ,1}(x)=\!\! \sup\limits_{\substack{ y\in M \\ \varphi (y^*y)\leq 1}}\!\!
{\rm Re}\,\varphi (y^*x)\;\! ,\; p_{\varphi ,2}(x)=\!\! \sup\limits_{\substack{ y\in M \\ \varphi (y\;\! y^*)\leq 1}}\!\!
{\rm Re}\,\varphi (x\;\! y^*)\;\! ,\qquad x\in M\;\! .$}
\medskip

Let $\varphi$ be a faithful, semi-finite, normal weight on a $W^*$-algebra $M$.
Defining, for $a\in (M^{\varphi})^+$, the semi-finite, normal weight ${\varphi}_a$ on $M$ by
\begin{equation}\label{derived}
{\varphi}_a(x)=\varphi (a^{1/2}xa^{1/2})\;\! ,\qquad x\in M^+\;\! ,
\end{equation}
if $A$ is a positive, self-adjoint operator affiliated to $M^{\varphi}$, then the formula
\medskip

\centerline{$\displaystyle {\varphi}_A(x)=\sup\limits_{\varepsilon >0}{\varphi}_{A(1+\varepsilon A)^{-1}}(x)
=\lim\limits_{0<\varepsilon\to 0}{\varphi}_{A(1+\varepsilon A)^{-1}}(x)\;\! ,\qquad x\in M^+$}

\noindent defines a $\sigma^{\varphi}$-invariant, semi-finite, normal weight ${\varphi}_A$
on $M$ and
\begin{equation*}
\sigma^{{\varphi}_A}_t(x)=A^{it}{\sigma}^{\varphi}_t(x)A^{-it}\;\! ,\qquad
t\in\mathbb{R}\;\! ,x\in s(A)Ms(A)
\end{equation*}
(see \cite{PT}, Section 4). We notice that the notation is not contradictory because for

\noindent bounded $A$ the above definition yields ${\varphi}_A(x)=\varphi (A^{1/2}xA^{1/2})$
for all $x\in M^+$.

Actually formula (\ref{derived}) defines a normal weight $\varphi_a$ for any normal weight
$\varphi$ on $M$ and any $a\in M^+$,
but if $\varphi$ is not semi-finite and
$a$ does not belong to $M^{\varphi}$, we cannot be sure that the weight $\varphi_a$
is semi-finite.

If $\varphi$ is a faithful, semi-finite, normal weight on a $W^*$-algebra $M$, and $A\;\! ,B$
are commuting positive, self-adjoint operators affiliated to $M^{\varphi}$, such that $s(B)\leq s(A)\;\!$,
then $B$ is affiliated also to $M^{{\varphi}_A}$ and $({\varphi}_A)_B={\varphi}_{\overline{AB}}$
(see \cite{PT}, Proposition 4.3).
\smallskip

Similarly as in the case of weights, if $N$ is a $W^*$-subalgebra of a $W^*$-algebra $M$, and
$E : M^+\longrightarrow \overline{N}\;\!^+$ an operator valued weight, then we use the notations
\begin{equation*}
\begin{split}
&\mathfrak{N}_{E}=\{ x\in M ; E (x^*x)\in N^+\}\;\! , \\
&\mathfrak{M}_{E}=\text{linear span of }\mathfrak{N}_{E}{\!\!\!\!}^*\,\mathfrak{N}_{E}
\end{split}
\end{equation*}
and notice that
\smallskip

\centerline{$\mathfrak{M}_{E}\cap M^+=\{ a\in M^+ ;E (a)\in N^+\}\;\! .$}
\medskip

Statement (\ref{reduced}) holds true also for operator valued weights : If
$E : M^+\longrightarrow \overline{N}\;\!^+$ is an operator valued weight and $e\in M$
is a projetion, then
\begin{equation}\label{reduced-oper.val.}
E(1_M-e)=0\;\!\Longrightarrow E(x^*x)=E(e\;\! x^*x\;\! e)\text{ for all }x\in M\;\! .
\end{equation}
Indeed,
by the very definition of the extended positive part of a $W^*$-algebra (\cite{H1},
Definition 1.1) and of operator valued weights (\cite{H1}, Definition 2.1), the equality
of the operator valued weights $E$ and $E(e\;\!\cdot\;\! e)$ means the equality
\smallskip

\centerline{$\varphi\big( E(x^*x)\big) =\varphi\big( E(e\;\! x^*x\;\! e)\big)\text{ for all }x
\in M\text{ and all }\varphi\in M_*^+\;\! ,$}
\smallskip

\noindent that is the equality, for every $\varphi\in M_*^+\;\!$, of the weights $\varphi\circ E$
and $(\varphi\circ E)(e\;\!\cdot\;\! e)\;\!$. But this is an imediate consequence of (\ref{reduced}).
\smallskip

The support $s(E)$ of a normal operator valued weight
$E : M^+\longrightarrow \overline{N}\;\!^+$ is $1_M-q\;\!$, where $q$ is the greatest
projection in $M$ satisfying $E(q)=0\;\!$, and it belongs to the

\noindent relative commutant $N'\cap M$
(see \cite{H1}, Definition 2.8).

If $E : M^+\longrightarrow \overline{N}\;\!^+$ is a faithful, semi-finite, normal operator valued
weight, then
\begin{itemize}
\item $E(\mathfrak{M}_{E})$ is a $w$-dense two-sided ideal in $N$ (\cite{H1}, Proposition 2.5).
\item For any faithful, semi-finite, normal weight $\varphi$ on $N$, the composition $\varphi\circ E$
is a faithful, semi-finite, normal weight on $M$ (\cite{H1}, Proposition 2.3) and
\smallskip

\centerline{$\qquad\qquad E\big(\sigma^{\varphi\circ E}_t(a)\big) =
\sigma^{\varphi}_t\big( E(a)\big)\;\! ,\qquad a\in M^+, t\in\mathbb{R}$}
\smallskip

\noindent (\cite{H1}, Proposition 4.9). In particular, $\sigma^{\varphi\circ E}_t(\mathfrak{M}_E)=
\mathfrak{M}_E$ for every $t\in\mathbb{R}\;\!$.
\end{itemize}
Moreover, $\varphi\circ E$ determines uniquely $E$ in the following sense:
\begin{itemize}
\item If $E_1 ,E_2 : M^+\longrightarrow \overline{N}\;\!^+$ are faithful, semi-finite, normal weights such that

\noindent $\varphi\circ E_1=\varphi\circ E_2$ for some faithful, semi-finite, normal weight $\varphi$ on $N$,
then $E_1=E_2$ (\cite{H1}, Lemma 4.8).
\end{itemize}




\section{Regularizing nets}

Let $M$ be a $W^*$-algebra, and $\varphi$ a faithful, semi-finite, normal weight on $M$.
We will call (slightly differently as in \cite{Z2}, \S 1), {\it regularizing net} for $\varphi$
any net $(u_{\iota})_{\iota}$ in $\mathfrak{T}_{\varphi}$ such that
\begin{itemize}
\item [(i)] $\sup\limits_{\substack{ z\in K \\ \iota }}\|\sigma^{\varphi}_z(u_{\iota})\|
<+\infty$ and $\sup\limits_{\substack{ z\in K \\ \iota }}\|\sigma^{\varphi}_z(u_{\iota})_{\varphi}\|
<+\infty$ for each compact $K\subset\mathbb{C}\;\!$;
\item[(ii)] $\sigma^{\varphi}_z(u_\iota) \xrightarrow{\;\; \iota\;\,}1_M$ in the $s^*$-topology
for all $z\in\mathbb{C}\;\!$.
\end{itemize}

Regularizing nets are useful in the modular theory of faithful, semi-finite, normal weights.
Usually they are constructed starting with a bounded net $(x_\iota )_{\iota}$ in $\mathfrak{A}_{\varphi}$
such that $x_{\iota} \xrightarrow{\;\; \iota\;\,}1_M$ in the $s^*$-topology and then getting it
"mollified", for example, by the mollifier $e^{-t^2}$, that is passing to the net $(u_\iota )_{\iota}$
with
\begin{equation}\label{mollify}
u_{\iota} =\frac 1{\;\!\sqrt{\pi}\;\!}\! \int\limits_{-\infty}^{+\infty}\!\!
e^{-t^2}\! {\sigma}^{\varphi}_t(x_{\iota})\;\!{\rm d}t\;\! .
\end{equation}
The verification of (i) is straightforward, more troublesome is to verify the inclusion
$u_{\iota}\in \mathfrak{T}_{\varphi}$ and the convergence (ii).

Concerning the verification of (ii),
if the net $(x_\iota )_{\iota}$ would be increasing, we could proceed as in the proof of
\cite{PT}, Lemma 5.2 by using Dini's theorem. But there are situations in which we cannot
restrict us to the case of increasing $(x_\iota )_{\iota}\;\!$.
For example,  it is not clear whether every
$s^*$-dense, $\sigma^{\varphi}$-invariant (not necessarily hereditary) ${}^*$-subalgebra of $M$
contains some increasing net $(x_\iota )_{\iota}$ with $x_{\iota} \xrightarrow{\;\; \iota\;\,}1_M$ in
the $s^*$-topology, as used in the proof of \cite{PT}, Lemma 5.2.

On the other hand, if the net $(x_\iota )_{\iota}$ would be 
a sequence, then we could use the dominated convergence theorem of Lebesgue.
similarly as, for example, in the proof of \cite{St}, Theorem 2.16.
But again, unless $M$ is countably decomposable (and so its unit ball $s^*$-metrizable),
the unit ball of not every $s^*$-dense ${}^*$-subalgebra of $M$ contains a sequence
$s^*$-convergent to $1_M\;\!$.

Therefore, in order to cover also the case of non countable nets $(x_\iota )_{\iota}\;\!$, we have
to verify (ii) directly, taking advantage of the particularities of the situation.

In this section we will prove that, starting with a bounded net $(x_\iota )_{\iota}$ even in
$\mathfrak{N}_{\varphi}\;\!$, formula (\ref{mollify}) furnishes a regularizing net $(u_\iota )_{\iota}\;\!$.

Let us begin with recalling some facts concerning the modular theory of faithful, semi-finite, normal
weights.
The next lemma is \cite{AZ}, (2.27) :

\begin{lemma}\label{integral}
Let $\varphi$ be a faithful, semi-finite, normal weight on a $W^*$-algebra $M$.

\noindent If $x\in\mathfrak{N}_{\varphi}$ and $f\in L^1(\mathbb{R})\;\!$, then
\smallskip

\centerline{$\displaystyle \int\limits_{-\infty}^{+\infty}\! f(t) \sigma^{\varphi}_t(x)\;\!{\rm d}t
\in\mathfrak{N}_{\varphi}\text{ and }\Big(\! \int\limits_{-\infty}^{+\infty}\! f(t)
\sigma^{\varphi}_t(x)\;\!{\rm d}t\Big)_{\! \varphi}\! =\! \int\limits_{-\infty}^{+\infty}\! f(t)
\Delta_{\varphi}^{it} x_{\varphi}\;\!{\rm d}t\;\! .$}
\end{lemma}

\hfill $\square$
\smallskip

Let $\varphi$ be a faithful, semi-finite, normal weight on a $W^*$-algebra $M$,
and $z\in\mathbb{C}\;\!$.

We define the linear operator $\sigma^{\varphi}_z : M\supset
\mathcal{D}(\sigma^{\varphi}_z)\ni x\longmapsto \sigma^{\varphi}_z(x)\in M$ as follows :
the pair $\big( x\;\! ,\sigma^{\varphi}_z(x)\big)$ belongs to its graph whenever the
map $\mathbb{R}\ni t\longmapsto \sigma^{\varphi}_t(x)\in M$ has a $w$-continuous
extension on the closed strip
\smallskip

\centerline{$\big\{ \zeta\in\mathbb{C} ; 0\leq |{\rm Im}\zeta |\leq |{\rm Im}z |\;\! ,
({\rm Im}\zeta )({\rm Im}z)\geq 0\big\}\;\! ,$}
\smallskip

\noindent analytic in the interior and taking the value $\sigma^{\varphi}_z(x)$ at $z\;\!$.

It is easily seen (see e.g. \cite{Z1}, Theorem 1.6) that, for each $z\in\mathbb{C}\;\!$,
\begin{equation}\label{star}
\mathcal{D}(\sigma^{\varphi}_z)^*=
\mathcal{D}(\sigma^{\varphi}_{\overline{z}})\text{ and }\sigma^{\varphi}_{\overline{z}}(x^*)=
\sigma^{\varphi}_z(x)^*\text{ for every }x\in \mathcal{D}(\sigma^{\varphi}_z)\;\! .
\end{equation}

We recall that $x\in M$ belongs to $\mathcal{D}(\sigma^{\varphi}_z)$ if and only if the operator
$\Delta_{\varphi}^{iz}\pi_{\varphi}(x)\Delta_{\varphi}^{-iz}$ is defined and bounded on a core of
$\Delta_{\varphi}^{-iz}\;\!$, in which case
\smallskip

\centerline{$\mathcal{D}\big( \Delta_{\varphi}^{iz}\pi_{\varphi}(x)\Delta_{\varphi}^{-iz} \big) =
\mathcal{D}(\Delta_{\varphi}^{-iz})\text{ and }\overline{\Delta_{\varphi}^{iz}\pi_{\varphi}(x)\Delta_{\varphi}^{-iz}}
=\pi_{\varphi}\big( \sigma^{\varphi}_z(x)\big)$}
\smallskip

\noindent that is
\begin{equation}\label{implement}
\pi_{\varphi}(x)\Delta_{\varphi}^{-iz}\subset \Delta_{\varphi}^{-iz}\pi_{\varphi}\big( \sigma^{\varphi}_z(x)\big)
\end{equation}
\noindent (see \cite{CZ}, Theorem 6.2 or \cite{AZ}, Theorem 2.3).
\smallskip

Now we prove a criterion for an element of $\mathfrak{N}_{\varphi}$ to belong to
$\mathfrak{N}_{\varphi}{\!\!\!}^*$, hence to $\mathfrak{A}_{\varphi}\;\!$:

\begin{lemma}\label{the-adjoint}
Let $\varphi$ be a faithful, semi-finite, normal weight on a $W^*$-algebra $M$,

\noindent and $x\in\mathfrak{N}_{\varphi}\;\!$. Then
\smallskip

\centerline{$x\in\mathcal{D}(\sigma^{\varphi}_{-i/2})\text{ and }\sigma^{\varphi}_{-i/2}(x)\in
\mathfrak{N}_{\varphi} \Longrightarrow
x\in \mathfrak{N}_{\varphi}{\!\!\!}^*\text{ and }\sigma^{\varphi}_{-i/2}(x)_{\varphi}=
\Delta_{\varphi}^{1/2} x_{\varphi}\;\! ,$}

\noindent\hspace{6.64 cm}that is

\noindent\hspace{6.64 cm}$x\in\mathfrak{A}_{\varphi}\text{ and }S_{\varphi} x_{\varphi}=
J_{\varphi} \sigma^{\varphi}_{-i/2}(x)_{\varphi}\;\! .$
\end{lemma}

\begin{proof}
Let $y\in\mathfrak{M}_{\varphi}$ be arbitrary. Then
\begin{equation}\label{first}
\pi_{\varphi}(x^*)J_{\varphi} y_{\varphi}=\pi_{\varphi}(x^*)J_{\varphi} \big( S_{\varphi}(y^*)_{\varphi}\big)
=\pi_{\varphi}(x^*)\Delta_{\varphi}^{1/2} (y^*)_{\varphi}
\end{equation}

\noindent Application of (\ref{star}) with $\displaystyle z= -\;\! \frac{i}2$ yields
$x^*\in\mathcal{D}(\sigma^{\varphi}_{i/2})$ and $\sigma^{\varphi}_{i/2}(x^*)=\sigma^{\varphi}_{-i/2}(x)^*$,

\noindent so, applying (\ref{implement}) to $x^*$ and $\displaystyle z= \frac{i}2\;\!$, we deduce
\begin{equation}\label{second}
\pi_{\varphi}(x^*)\Delta_{\varphi}^{1/2}\subset \Delta_{\varphi}^{1/2}
\pi_{\varphi}\big( \sigma^{\varphi}_{i/2}(x^*)\big) =\Delta_{\varphi}^{1/2}
\pi_{\varphi}\big( \sigma^{\varphi}_{-i/2}(x)^*\big)\;\! .
\end{equation}
By (\ref{first}) and (\ref{second}) we conclude :
\begin{equation*}
\begin{split}
\pi_{\varphi}(x^*)J_{\varphi} y_{\varphi}=\;&\Delta_{\varphi}^{1/2}\pi_{\varphi}\big( \sigma^{\varphi}_{-i/2}(x)^*\big)
(y^*)_{\varphi} =\Delta_{\varphi}^{1/2}\big( \sigma^{\varphi}_{-i/2}(x)^*y^*\big)_{\varphi} \\
=\;&J_{\varphi}S_{\varphi} \Big(\big( y\;\! \sigma^{\varphi}_{-i/2}(x)\big)^*\Big)_{\varphi}
=J_{\varphi} \big( y\;\! \sigma^{\varphi}_{-i/2}(x)\big)_{\varphi} \\
=\;&J_{\varphi} \pi_{\varphi}(y)  \sigma^{\varphi}_{-i/2}(x)_{\varphi}\;\! .
\end{split}
\end{equation*}

By the aboves
\smallskip

\centerline{$\| \pi_{\varphi}(x^*)J_{\varphi} y_{\varphi}\|\leq\| \sigma^{\varphi}_{-i/2}(x)_{\varphi}\|\cdot \| y\|\;\! ,
\qquad y\in\mathfrak{M}_{\varphi}\;\! ,$}
\smallskip

\noindent so we can apply \cite{AZ}, Lemma 2.6 (1) to deduce that $x^*\in\mathfrak{N}_{\varphi}
\Longleftrightarrow x\in\mathfrak{N}_{\varphi}{\!\!\!}^*\;\!$.
\smallskip

Taking into account that $x\in\mathfrak{A}_{\varphi}$ and $y\in \mathfrak{M}_{\varphi}\subset
\mathfrak{A}_{\varphi}\;\!$, and using \cite{AZ}, (2.5), as well as the above
(\ref{implement}) with $\displaystyle z= -\;\! \frac{i}2\;\!$, we deduce :
\begin{equation*}
\begin{split}
\pi_{\varphi}(y) J_{\varphi} \sigma^{\varphi}_{-i/2}(x)_{\varphi} =\;&
J_{\varphi} \pi_{\varphi}\big(\sigma^{\varphi}_{-i/2}(x)\big) J_{\varphi} y_{\varphi} =
J_{\varphi} \pi_{\varphi}\big(\sigma^{\varphi}_{-i/2}(x)\big) J_{\varphi} S_{\varphi} (y^*)_{\varphi} \\
=\;&J_{\varphi} \pi_{\varphi}\big(\sigma^{\varphi}_{-i/2}(x)\big) \Delta_{\varphi}^{1/2} (y^*)_{\varphi}
=J_{\varphi} \Delta_{\varphi}^{1/2} \pi_{\varphi}(x) (y^*)_{\varphi} \\
=\;&S_{\varphi} (x\;\! y^*)_{\varphi} =(y\;\! x^*)_{\varphi} =\pi_{\varphi}(y) (x^*)_{\varphi}
=\pi_{\varphi}(y) S_{\varphi} x_{\varphi} \\
=\;&\pi_{\varphi}(y) J_{\varphi} \Delta_{\varphi}^{1/2} x_{\varphi}\;\! .
\end{split}
\end{equation*}
Since $\pi_{\varphi}(\mathfrak{M}_{\varphi})$ is $w$-dense in $M$, it follows the equality
$\sigma^{\varphi}_{-i/2}(x)_{\varphi} = \Delta_{\varphi}^{1/2} x_{\varphi}\;\! .$

\end{proof}

The above two lemmas can be used to produce elements of the Tomita algebra $\mathfrak{T}_{\varphi}$
by "regularizing" elements of $\mathfrak{N}_{\varphi}$ (not only elements of $\mathfrak{A}_{\varphi}\;\!$,
as customary : see in \cite{SZ1} the comments after the proof of Theorem 10.20 on page 347) :

\begin{lemma}\label{regularize}
Let $\varphi$ be a faithful, semi-finite, normal weight on a $W^*$-algebra $M$.
For each $x\in M$,

\centerline{$\displaystyle u =\frac 1{\;\!\sqrt{\pi}\;\!}\! \int\limits_{-\infty}^{+\infty}\!\!
e^{-t^2}\! {\sigma}^{\varphi}_t(x)\;\!{\rm d}t$}
\smallskip

\noindent  belongs to $M^{\varphi}_{\infty}$ and
\begin{equation}\label{formula}
\sigma^{\varphi}_z(u) =\frac 1{\;\!\sqrt{\pi}\;\!}\! \int\limits_{-\infty}^{+\infty}\!\!
e^{-(t-z)^2}\! {\sigma}^{\varphi}_{t}(x)\;\!{\rm d}t\;\! ,\qquad z\in\mathbb{C}\;\! .
\end{equation}
Assuming that $x\in\mathfrak{N}_{\varphi}\;\!$, we have $u\in \mathfrak{T}_{\varphi}\;\!$.
\end{lemma}

\begin{proof}
Since
\smallskip

\centerline{$\displaystyle \mathbb{R}\ni s\longmapsto \sigma^{\varphi}_s(u) =
\frac 1{\;\!\sqrt{\pi}\;\!}\! \int\limits_{-\infty}^{+\infty}\!\!
e^{-t^2}\! {\sigma}^{\varphi}_{s+t}(x)\;\!{\rm d}t =
\frac 1{\;\!\sqrt{\pi}\;\!}\! \int\limits_{-\infty}^{+\infty}\!\!
e^{-(t-s)^2}\! {\sigma}^{\varphi}_{t}(x)\;\!{\rm d}t$}
\smallskip

\noindent allows the entire extension
\smallskip

\centerline{$\displaystyle \mathbb{C}\ni z\longmapsto
\frac 1{\;\!\sqrt{\pi}\;\!}\! \int\limits_{-\infty}^{+\infty}\!\!
e^{-(t-z)^2}\! {\sigma}^{\varphi}_{t}(x)\;\!{\rm d}t\;\! ,$}
\smallskip

\noindent we have $u\in M^{\varphi}_{\infty}$ and (\ref{formula}) holds true.

It remains to show that, assuming $x\in\mathfrak{N}_{\varphi}\;\!$, we have
\smallskip

\centerline{$\sigma^{\varphi}_z(u)\in \mathfrak{A}_{\varphi}\;\! ,\qquad z\in\mathbb{C}\;\! .$}
\smallskip

Using (\ref{formula}) it is easy to see that
\medskip

\centerline{$\sigma^{\varphi}_t\big( \sigma^{\varphi}_z(u)\big) =\sigma^{\varphi}_{z+t}(u)\;\! ,
\qquad z\in\mathbb{C}\;\! ,t\in\mathbb{R}\;\! ,$}

\noindent so
\begin{equation}\label{semigroup}
\sigma^{\varphi}_z(u)\in M^{\varphi}_{\infty}\text{ and }
\sigma^{\varphi}_{\zeta}\big( \sigma^{\varphi}_z(u)\big) =\sigma^{\varphi}_{z+\zeta}(u)\;\! ,
\qquad z\;\! ,\zeta\in\mathbb{C}\;\! ,
\end{equation}

For each $z\in\mathbb{C}\;\!$, applying Lemma \ref{integral} with
$\displaystyle f(t)=\frac 1{\;\!\sqrt{\pi}\;\!}\;\! e^{-(t-z)^2}$, we deduce that

\noindent $\sigma^{\varphi}_z(u)\in  \mathfrak{N}_{\varphi}\;\!$.
Since $z\in\mathbb{C}$ is here arbitrary, also $\sigma^{\varphi}_{z-i/2}(u)\in  \mathfrak{N}_{\varphi}$
holds true. But by

\noindent (\ref{semigroup}) we have $\sigma^{\varphi}_{-i/2}\big( \sigma^{\varphi}_z(u)\big) =
\sigma^{\varphi}_{z-i/2}(u)\;\!$, so $\sigma^{\varphi}_{-i/2}\big( \sigma^{\varphi}_z(u)\big)\in
\mathfrak{N}_{\varphi}\;\!$. Applying now

\noindent Lemma \ref{the-adjoint}, we conclude that
$\sigma^{\varphi}_z(u)$ belongs also to $\mathfrak{N}_{\varphi}{\!\!\!}^*$, hence
$\sigma^{\varphi}_z(u)\in\mathfrak{A}_{\varphi}\;\!$.

\end{proof}

Next we prove a dominated convergence theorem for integrals of the form (\ref{formula})
and nets of arbitrary cardinality :

\begin{lemma}\label{dom.conv.}
Let $\varphi$ be a faithful, semi-finite, normal weight on a $W^*$-algebra $M$,
and $(x_{\iota})_{\iota}$ a net in the closed unit ball of $M$ such that
$x_\iota \xrightarrow{\; \iota\;}1_M$ in the $s^*$-topology.
Let the net $(u_{\iota})_{\iota}$ be defined by the formula
\smallskip

\centerline{$\displaystyle u_{\iota} =\frac 1{\;\!\sqrt{\pi}\;\!}\! \int\limits_{-\infty}^{+\infty}\!\!
e^{-t^2}\! {\sigma}^{\varphi}_t(x_{\iota})\;\!{\rm d}t\;\! .$}
\smallskip

\noindent Then
\begin{itemize}
\item[(i)] $u_{\iota}\in M^{\varphi}_{\infty}$ for all $\iota\;\!$;
\item[(ii)] $\|\sigma^{\varphi}_z(u_\iota)\|\leq e^{({\rm Im}z)^2}$ for all $\iota$ and $z\in\mathbb{C}\;\!$;
\item[(iii)] $\sigma^{\varphi}_z(u_\iota) \xrightarrow{\;\; \iota\;\,}1_M$ in the $s^*$-topology
for all $z\in\mathbb{C}\;\!$.
\end{itemize}
\end{lemma}

\begin{proof}
(i) is an immediate consequence of Lemma \ref{regularize}.

For (ii), let $\iota$ and $z\in\mathbb{C}$ be arbitrary. By Lemma \ref{regularize} we have
\begin{equation}\label{int.repr.}
\sigma^{\varphi}_z(u_{\iota}) =\frac 1{\;\!\sqrt{\pi}\;\!}\! \int\limits_{-\infty}^{+\infty}\!\!
e^{-(t-z)^2}\! {\sigma}^{\varphi}_{t}(x_{\iota})\;\!{\rm d}t\;\! .
\end{equation}
Since $\| {\sigma}^{\varphi}_{t}(x_{\iota})\| =\| x_{\iota}\|\leq 1$ for all $t\in\mathbb{R}\;\!$,
it follows
\begin{equation*}
\| \sigma^{\varphi}_z(u_{\iota})\|\leq \frac 1{\;\!\sqrt{\pi}\;\!}\! \int\limits_{-\infty}^{+\infty}\!\!
|\;\! e^{-(t-z)^2}|\;\!{\rm d}t =\frac 1{\;\!\sqrt{\pi}\;\!}\! \int\limits_{-\infty}^{+\infty}\!\!
e^{-(t-{\rm Re}z)^2+({\rm Im}z)^2}{\rm d}t =e^{({\rm Im}z)^2}\;\! .
\end{equation*}

The more involved issue is (iii). For fixed $z\in\mathbb{C}\;\!$, we have to show that
\smallskip

\centerline{$\displaystyle \sigma^{\varphi}_z(u_\iota)-1_M
\overset{(\ref{int.repr.})}{=}\frac 1{\;\!\sqrt{\pi}\;\!}\! \int\limits_{-\infty}^{+\infty}\!\!
e^{-(t-z)^2}\! {\sigma}^{\varphi}_{t}(x_{\iota}-1_M)\;\!{\rm d}t \xrightarrow{\;\; \iota\;\,} 0$}
\smallskip

\noindent in the $s^*$-topology. Since the $s^*$-topology is defined by the semi-norms
\medskip

\centerline{$p_{\psi} : M\ni x\longmapsto \psi (x^*x)^{1/2}+ \psi (x\;\! x^*)^{1/2}\;\!$, $\psi$ a
normal positive form on $M\;\! ,$}
\medskip

\noindent this means that

\centerline{$\displaystyle p_{\psi}\bigg( \frac 1{\;\!\sqrt{\pi}\;\!}\! \int\limits_{-\infty}^{+\infty}\!\!
e^{-(t-z)^2}\! {\sigma}^{\varphi}_{t}(x_{\iota}-1_M)\;\!{\rm d}t\bigg) \xrightarrow{\;\; \iota\;\,} 0$}
\smallskip

\noindent for every normal positive form $\psi$ on $M\;\!$.
\smallskip

For let $\psi$ be any normal positive form on $M\;\!$. Since, according to (\ref{semi-norm-ineq}),
\begin{equation*}
\begin{split}
p_{\psi}\bigg( \frac 1{\;\!\sqrt{\pi}\;\!}\! \int\limits_{-\infty}^{+\infty}\!\!
e^{-(t-z)^2}\! {\sigma}^{\varphi}_{t}(x_{\iota}-1_M)\;\!{\rm d}t\bigg) \leq\;&
\frac 1{\;\!\sqrt{\pi}\;\!}\! \int\limits_{-\infty}^{+\infty}\!\!
p_{\psi}\Big( e^{-(t-z)^2}\! {\sigma}^{\varphi}_{t}(x_{\iota}-1_M) \Big)\;\!{\rm d}t \\
=\;&\frac 1{\;\!\sqrt{\pi}\;\!}\! \int\limits_{-\infty}^{+\infty}\!
\Big|\;\! e^{-(t-z)^2}\Big|\;\! p_{\psi}\big( {\sigma}^{\varphi}_{t}(x_{\iota}-1_M) \big)\;\!{\rm d}t\;\! ,
\end{split}
\end{equation*}
the proof is done if we prove the convergence
\smallskip

\centerline{$\displaystyle \int\limits_{-\infty}^{+\infty}\!
\Big|\;\! e^{-(t-z)^2}\Big|\;\! p_{\psi}\big( {\sigma}^{\varphi}_{t}(x_{\iota}-1_M) \big)\;\!{\rm d}t
\xrightarrow{\;\; \iota\;\,} 0\;\! ,$}
\smallskip

\noindent which of course is consequence of
\begin{equation}\label{convergence}
\int\limits_{-\infty}^{+\infty}\!\!
e^{-(t-{\rm Re}z)^2}\! (\psi\circ {\sigma}^{\varphi}_{t}) \Big( (x_{\iota}-1_M)^*(x_{\iota}-1_M)+
(x_{\iota}-1_M)\;\! (x_{\iota}-1_M)^*\Big)^{1/2}\! {\rm d}t
\xrightarrow{\;\; \iota\;\,} 0\;\! ,
\end{equation}
because $\displaystyle \Big|\;\! e^{-(t-z)^2}\Big| =e^{-(t-{\rm Re}z)^2+({\rm Im}z)^2}$ and
\begin{equation*}
\begin{split}
&p_{\psi}\big( {\sigma}^{\varphi}_{t}(x_{\iota}-1_M) \big) \\
=\;&(\psi\circ {\sigma}^{\varphi}_{t})\big( (x_{\iota}-1_M)^*(x_{\iota}-1_M)\big)^{1/2}
+(\psi\circ {\sigma}^{\varphi}_{t})\big( (x_{\iota}-1_M)\;\! (x_{\iota}-1_M)^*\big)^{1/2} \\
\leq\;& \sqrt{2}\;\! (\psi\circ {\sigma}^{\varphi}_{t}) \Big( (x_{\iota}-1_M)^*(x_{\iota}-1_M)
+(x_{\iota}-1_M)\;\! (x_{\iota}-1_M)^*\Big)^{1/2}
\end{split}
\end{equation*}
We go to complete the proof by verifying (\ref{convergence}).
\smallskip

Since $x_\iota \xrightarrow{\; \iota\;}1_M$ in the $s^*$-topology and $\| x_{\iota}\|\leq 1$
for all $\iota\;\!$, we have that
\smallskip

\centerline{$\Big( (x_{\iota}-1_M)^*(x_{\iota}-1_M)
+(x_{\iota}-1_M)\;\! (x_{\iota}-1_M)^*\Big)_{\iota}$}
\smallskip

\noindent is a bounded net, convergent to $0$ in the $s^*$-topology.
According to a theorem due to C. A. Akemann (see \cite{Ak}, Theorem II.7 or
\cite{SZ2}, Corollary 8.17), on bounded subsets of $M$ the $s^*$-topology coincides with
the Mackey topology $\tau_w$ associated to the $w$-topology, that is with the topology of the
uniform convergence on the weakly compact absolutely convex subsets of the predual $M_*\;\!$.
Since, by the classical Krein-\u Smulian theorem (see e.g. \cite{DS}, Theorem V.6.4), the closed
absolutely convex hull of every weakly compact set in a Banach space is still weakly compact,
$\tau_w$ is actually the topology of the uniform convergence on the weakly compact subsets of
$M_*\;\!$. Therefore
\begin{equation}\label{Mackey}
\sup\limits_{\theta\in K} \Big|\;\! \theta \Big( (x_{\iota}-1_M)^*(x_{\iota}-1_M)+
(x_{\iota}-1_M)\;\! (x_{\iota}-1_M)^*\Big) \Big| \xrightarrow{\; \iota\;} 0
\end{equation}
for every weakly compact $K\subset M_*\;\!$.
\smallskip

Now let $\varepsilon >0$ be arbitrary. Choose some $t_0>0$ such that
\begin{equation}\label{big-t}
\int\limits_{|t|>t_0}\!\!\! e^{-(t-{\rm Re}z)^2}{\rm d}t\leq\frac{\varepsilon}{\;\! 4 \sqrt{2\;\! \|\psi\|}\;\!}\, .
\end{equation}
Since $K_{t_0}=\{ \psi\circ {\sigma}^{\varphi}_{t} ; |t|\leq t_0\}$ is a weakly compact
subset of $M_*\;\!$, (\ref{Mackey}) holds true with $K=K_{t_0}\;\!$. Thus there exists
some $\iota_0$ such that
\begin{equation}\label{small-t}
\sup\limits_{|t|\leq t_0} \Big|\;\! (\psi\circ {\sigma}^{\varphi}_{t}) \Big( (x_{\iota}-1_M)^*(x_{\iota}-1_M)+
(x_{\iota}-1_M)\;\! (x_{\iota}-1_M)^*\Big) \Big| \leq \frac{\varepsilon}{\;\! 2\;\! \sqrt{\pi}\;\!}
\end{equation}
for all $\iota\geq\iota_0\;\!$. (\ref{big-t}) implies
\begin{equation*}
\begin{split}
&\int\limits_{|t|>t_0}\!\!\!
e^{-(t-{\rm Re}z)^2}\! (\psi\circ {\sigma}^{\varphi}_{t}) \Big( (x_{\iota}-1_M)^*(x_{\iota}-1_M)+
(x_{\iota}-1_M)\;\! (x_{\iota}-1_M)^*\Big)^{1/2}\! {\rm d}t \\
\leq\;&\int\limits_{|t|>t_0}\!\!\! e^{-(t-{\rm Re}z)^2}\! \big( 8\;\! \|\psi\|\big)^{1/2}{\rm d}t
\leq \frac{\varepsilon}{\;\! 4 \sqrt{2\;\! \|\psi\|}\;\!} \big( 8\;\! \|\psi\|\big)^{1/2} =\frac{\varepsilon}{\;\! 2\;\!}\;\! ,
\end{split}
\end{equation*}
while using (\ref{small-t}) we deduce for every $\iota\geq\iota_0\;\!$:
\begin{equation*}
\begin{split}
&\int\limits_{|t|\leq t_0}\!\!\!
e^{-(t-{\rm Re}z)^2}\! (\psi\circ {\sigma}^{\varphi}_{t}) \Big( (x_{\iota}-1_M)^*(x_{\iota}-1_M)+
(x_{\iota}-1_M)\;\! (x_{\iota}-1_M)^*\Big)^{1/2}\! {\rm d}t \\
\leq\;&\int\limits_{|t|\leq t_0}\!\!\! e^{-(t-{\rm Re}z)^2}\! \frac{\varepsilon}{\;\! 2\;\! \sqrt{\pi}\;\!}\;\!{\rm d}t
\leq \frac{\varepsilon}{\;\! 2\;\! \sqrt{\pi}\;\!} \int\limits_{-\infty}^{+\infty}\!\! e^{-(t-{\rm Re}z)^2}{\rm d}t
=\frac{\varepsilon}{\;\! 2\;\! \sqrt{\pi}\;\!} \int\limits_{-\infty}^{+\infty}\!\! e^{-\;\! t^2}{\rm d}t
=\frac{\varepsilon}{\;\! 2\;\!}\;\! .
\end{split}
\end{equation*}
Consequently, for every $\iota\geq\iota_0\;\!$,
\begin{equation*}
\begin{split}
&\int\limits_{-\infty}^{+\infty}\!\!
e^{-(t-{\rm Re}z)^2}\! (\psi\circ {\sigma}^{\varphi}_{t}) \Big( (x_{\iota}-1_M)^*(x_{\iota}-1_M)+
(x_{\iota}-1_M)\;\! (x_{\iota}-1_M)^*\Big)^{1/2}\! {\rm d}t \\
=\;&\int\limits_{|t|>t_0}\!\!\! ...\; +\int\limits_{|t|\leq t_0}\!\!\! ...\; \leq \frac{\varepsilon}{\;\! 2\;\!}
+\frac{\varepsilon}{\;\! 2\;\!} =\varepsilon\;\! .
\end{split}
\end{equation*}
\end{proof}

\begin{remark}
The proof of statement (iii) in Lemma \ref{dom.conv.} was rather cumbersome, we could not avoid to use the deep
result of Akemann concerning $s^*$-topology.

As we noticed at the beginning of this section,
if $(x_{\iota})_{\iota}$ would be an increasing net of positive elements, then we
could make use of Dini's theorem, as it was done in the proof of \cite{PT}, Lemma 5.2.
On the other hand, if the net $(x_{\iota})_{\iota}$ would be a

\noindent sequence, then we could use the Lebesgue dominated convergence theorem,
\end{remark}

Lemmas \ref{dom.conv.} and \ref{regularize} yield immediately :

\begin{theorem}\label{reg.net}
Let $\varphi$ be a faithful, semi-finite, normal weight on a $W^*$-algebra $M$,
and $(x_{\iota})_{\iota}$ a net in the closed unit ball of $M$ such that
$x_\iota \xrightarrow{\; \iota\;}1_M$ in the $s^*$-topology.
Let the net $(u_{\iota})_{\iota}$ be defined by the formula
\smallskip

\centerline{$\displaystyle u_{\iota} =\frac 1{\;\!\sqrt{\pi}\;\!}\! \int\limits_{-\infty}^{+\infty}\!\!
e^{-t^2}\! {\sigma}^{\varphi}_t(x_{\iota})\;\!{\rm d}t\;\! .$}
\smallskip

\noindent Then
\smallskip

\begin{itemize}
\item[(i)] $u_{\iota}\in M^{\varphi}_{\infty}$ for all $\iota\;\!$;
\item[(ii)] $\|\sigma^{\varphi}_z(u_\iota)\|\leq e^{({\rm Im}z)^2}$ for all $\iota$ and $z\in\mathbb{C}\;\!$;
\item[(iii)] $\sigma^{\varphi}_z(u_\iota) \xrightarrow{\;\; \iota\;\,}1_M$ in the $s^*$-topology
for all $z\in\mathbb{C}\;\!$.
\end{itemize}
\smallskip

\noindent Moreover, if $x_{\iota}\in\mathfrak{N}_{\varphi}$ for all $\iota\;\!$, then $u_{\iota}$
belongs to $\mathfrak{T}_{\iota}$ for every $\iota$ and therefore $(u_{\iota})_{\iota}$
is a regularizing net for $\varphi\;\!$.
\end{theorem}

\hfill $\square$
\smallskip

\section{Inequalities and equality between weights}


The main tool by proving criteria for inequalities and equalities between weights is
the following generalization of \cite{Z2}, Lemma 2.1.

We recall that a ${}^*$-subalgebra $\mathfrak{M}$ of a $W^*$-algebra $M$ is called
{\it facial subalgebra} or {\it hereditary subalgebra} whenever $\mathfrak{M}\cap M^+$%
is a face, that is a convex cone satisfying
\smallskip

\centerline{$M^+\ni b\leq a\in \mathfrak{M}\cap M^+\Longrightarrow b\in\mathfrak{M}\cap M^+\;\! ,$}
\smallskip

\noindent and $\mathfrak{M}$ is the linear span of it (see e.g. \cite{SZ1}, Section 3.21).

\begin{theorem}\label{ineq}
Let $M$ be a $W^*$-algebra, $\varphi$ a faithful, semi-finite, normal weight on $M$, $a\in (M^{\varphi})^+$,
and $\psi$ a normal weight on $M$. Assume that there exists a $w$-dense,

\noindent $\sigma^{\varphi}$-invariant ${}^*$-subalgebra $\mathfrak{M}$ of $\mathfrak{M}_{{\varphi}_a}$
such that
\begin{equation*}
\psi (x^*x)={\varphi}_a(x^*x)\;\! ,\qquad x\in\mathfrak{M}\;\! .
\end{equation*}
Then
\begin{equation}\label{ineq-phi_a}
\psi\leq{\varphi}_a\;\! .
\end{equation}
Additionally, there exists a $\sigma^{\varphi}$-invariant, hereditary ${}^*$-subalgebra $\mathfrak{M}_0$
of $\mathfrak{M}_{{\varphi}_a}$ such

\noindent that
\smallskip

\noindent\hspace{3.8 cm}$\mathfrak{M}\cap M^+\subset \mathfrak{M}_0\cap M^+\;\! ,$
\smallskip

\centerline{$\psi (b)={\varphi}_a (b)\;\! ,\qquad b\in \mathfrak{M}_0\cap M^+\;\! .$}
\end{theorem}

The difference between the above Theorem \ref{ineq} and \cite{Z2}, Lemma 2.1 consists in the fact that in
\cite{Z2}, Lemma 2.1 is additionally assumed that
\begin{itemize}
\item[(i)] $\psi$ is semi-finite and $\sigma^{\varphi}$-invariant and
\item[(ii)] $\mathfrak{M}$ is contained already in $\mathfrak{M}_{\varphi}$ (which of course, according to
\cite{PT}, Theorem 3.6, is a subset of $\mathfrak{M}_{{\varphi}_a}$).
\end{itemize}
However the proof of \cite{Z2}, Lemma 2.1 does not use assumption (i) and, on the other hand, we can
adapt it to work with the assumption $\mathfrak{M}\subset\mathfrak{M}_{{\varphi}_a}$ instead of
$\mathfrak{M}\subset\mathfrak{M}_{\varphi}\;\!$.

\begin{proof}
Let $x\in\mathfrak{M}\subset\mathfrak{M}_{{\varphi}_a}$ be arbitrary. Since $\psi (x^*x)={\varphi}_a(x^*x)<+\infty\;\!$,
we have

\noindent $x\in\mathfrak{N}_{\psi}\cap\mathfrak{N}_{{\varphi}_a}$ and therefore $\psi (x^*\cdot\;\! x)$ and
${\varphi}_a(x^*\cdot\;\! x)$ are normal positive forms on $M$. Taking into account that
\medskip

\centerline{$\psi (x^*y^*yx)={\varphi}_a(x^*y^*yx)\;\! ,\qquad y\in\mathfrak{M}$}
\medskip

\noindent and $\mathfrak{M}$ is $w$-dense in $M$, we deduce that
\begin{equation}\label{left-right}
\psi (x^*\cdot\;\! x)={\varphi}_a(x^*\cdot\;\! x)\;\! .
\end{equation}

Now, by the Kaplansky density theorem there exists a net $(a_\iota )_{\iota}$ in $\mathfrak{M}$ such that
$0\leq a_\iota \leq 1_M$ for all $\iota$ and $a_\iota \xrightarrow{\; s^*\,}1_M\;\!$. Set, for each $\iota\;\!$,
\begin{equation}\label{u}
u_\iota =\frac 1{\;\!\sqrt{\pi}\;\!}\! \int\limits_{-\infty}^{+\infty}\!\!
e^{-t^2}\! {\sigma}^{\varphi}_t(a_\iota )\;\!{\rm d}t\in M^+\;\! .
\end{equation}
Clearly, $0\leq u_\iota \leq 1_M$ for all $\iota\;\!$. According to Lemma \ref{dom.conv.},
$u_{\iota}\in M^{\varphi}_{\infty}$ for all $\iota$ and
\begin{equation}\label{conv.u}
\sigma^{\varphi}_z(u_\iota) \xrightarrow{\;\; \iota\;\,}1_M\text{ in the }s^*\text{-topology for all }z\in\mathbb{C}\;\! .
\end{equation}

Since $a^{1/2}\in M^{\varphi}$, also
\begin{equation}\label{ua}
u_\iota a^{1/2} =\frac 1{\;\!\sqrt{\pi}\;\!}\! \int\limits_{-\infty}^{+\infty}\!\!
e^{-t^2}\! {\sigma}^{\varphi}_t(a_\iota )\;\! a^{1/2}\;\!{\rm d}t =
\frac 1{\;\!\sqrt{\pi}\;\!}\! \int\limits_{-\infty}^{+\infty}\!\!
e^{-t^2}\! {\sigma}^{\varphi}_t(a_\iota a^{1/2})\;\!{\rm d}t
\end{equation}
belongs to $M^{\varphi}_{\infty}$ for each $\iota\;\!$.
Furthermore, $a_{\iota}\in\mathfrak{M}\subset\mathfrak{M}_{{\varphi}_a}$ yields
\smallskip

\centerline{$\varphi \big( (a_{\iota}a^{1/2})^*(a_{\iota}a^{1/2})\big) ={\varphi}_a(a_{\iota}^2)<+\infty\;\! ,$}
\smallskip

\noindent hence $a_{\iota}a^{1/2}\in \mathfrak{N}_{\varphi}\;\!$. Taking into account (\ref{ua})
and applying Lemma \ref{regularize},
we deduce that $u_\iota a^{1/2}\in \mathfrak{T}_{\varphi}$ for all $\iota\;\!$.
\smallskip

Let $y\in M$ and $\iota$ be arbitrary. Since $a_{\iota}\in\mathfrak{M}$ and $\mathfrak{M}$
is $\sigma^{\varphi}$-invariant, application of (\ref{left-right}) yields for every $t\;\! ,s\in\mathbb{R}$
and $k=0\;\! ,1\;\! ,2\;\! ,3\;\!$:
\begin{equation*}
\begin{split}
\psi &\Big( \big( \sigma^{\varphi}_t(a_{\iota})+i^k\sigma^{\varphi}_s(a_{\iota})\big)^* y^*y
\big( \sigma^{\varphi}_t(a_{\iota})+i^k\sigma^{\varphi}_s(a_{\iota})\big)\Big) \\
=\, {\varphi}_a &\Big( \big( \sigma^{\varphi}_t(a_{\iota})+i^k\sigma^{\varphi}_s(a_{\iota})\big)^* y^*y
\big( \sigma^{\varphi}_t(a_{\iota})+i^k\sigma^{\varphi}_s(a_{\iota})\big)\Big)\;\! .
\end{split}
\end{equation*}
Applying (\ref{integration}) with
\medskip

\centerline{$\displaystyle F(t\;\! ,s)=\frac 1{\;\!\pi\;\!}\;\! e^{-t^2-s^2}
\big( \sigma^{\varphi}_t(a_{\iota})+i^k\sigma^{\varphi}_s(a_{\iota})\big)^* y^*y
\big( \sigma^{\varphi}_t(a_{\iota})+i^k\sigma^{\varphi}_s(a_{\iota})\big)\;\! ,$}
\medskip

\noindent it follows for $k=0\;\! ,1\;\! ,2\;\! ,3\;\!$:
\begin{equation*}
\begin{split}
\psi &\bigg( \frac 1{\;\!\pi\;\!}\! \int\limits_{-\infty}^{+\infty} \int\limits_{-\infty}^{+\infty}\!
e^{-t^2-s^2} \big( \sigma^{\varphi}_t(a_{\iota})+i^k\sigma^{\varphi}_s(a_{\iota})\big)^* y^*y
\big( \sigma^{\varphi}_t(a_{\iota})+i^k\sigma^{\varphi}_s(a_{\iota})\big)\;\!{\rm d}t\;\!{\rm d}s \bigg) \\
=\, {\varphi}_a &\bigg( \frac 1{\;\!\pi\;\!}\! \int\limits_{-\infty}^{+\infty} \int\limits_{-\infty}^{+\infty}\!
e^{-t^2-s^2} \big( \sigma^{\varphi}_t(a_{\iota})+i^k\sigma^{\varphi}_s(a_{\iota})\big)^* y^*y
\big( \sigma^{\varphi}_t(a_{\iota})+i^k\sigma^{\varphi}_s(a_{\iota})\big)\;\!{\rm d}t\;\!{\rm d}s \bigg)\;\! .
\end{split}
\end{equation*}
Since, by (\ref{u}),
\begin{equation*}
\begin{split}
&u_{\iota}y^*yu_{\iota}=\frac 1{\;\!\pi\;\!}\! \int\limits_{-\infty}^{+\infty} \int\limits_{-\infty}^{+\infty}\!
e^{-t^2-s^2} \sigma^{\varphi}_s(a_{\iota}) y^*y  \sigma^{\varphi}_t(a_{\iota})\;\!{\rm d}t\;\!{\rm d}s \\
=\;&\frac 1{\;\! 4\;\!}\sum\limits_{k=0}^3 \frac{\;\! i^k\;\!}{\pi}
\int\limits_{-\infty}^{+\infty} \int\limits_{-\infty}^{+\infty}\!
e^{-t^2-s^2} \big( \sigma^{\varphi}_t(a_{\iota})+i^k\sigma^{\varphi}_s(a_{\iota})\big)^* y^*y
\big( \sigma^{\varphi}_t(a_{\iota})+i^k\sigma^{\varphi}_s(a_{\iota})\big)\;\!{\rm d}t\;\!{\rm d}s\;\! ,
\end{split}
\end{equation*}
we conclude that
\begin{equation}\label{left-right-u}
\psi (u_{\iota}y^*yu_{\iota})=\varphi_a(u_{\iota}y^*yu_{\iota})\;\! .
\end{equation}

Next let $y\in\mathfrak{N}_{\varphi}$ be arbitrary. Using (\ref{left-right-u}) and applying \cite{Co2},
Lemme 7 (b) or \cite{Z2}, Proposition 1.1, we deduce for every $\iota$
\begin{equation*}
\begin{split}
\psi (u_{\iota}y^*yu_{\iota})=\;&\varphi_a(u_{\iota}y^*yu_{\iota})
=\varphi( a^{1/2}u_{\iota}y^*yu_{\iota}a^{1/2})
=\| (yu_{\iota}a^{1/2})_{\varphi}\|^2 \\
=\;&\| J_{\varphi} \pi_{\varphi}\big( \sigma^{\varphi}_{- i/2}(a^{1/2}u_{\iota})\big) J_{\varphi} y_{\varphi}\|^2
=\| J_{\varphi} \pi_{\varphi}(a^{1/2}) \pi_{\varphi}\big(\sigma^{\varphi}_{- i/2}(u_{\iota})\big)
J_{\varphi} y_{\varphi}\|^2
\end{split}
\end{equation*}
Since $u_{\iota}y^*yu_{\iota} \xrightarrow{\;\; \iota\;\,} y^*y$ and
$\sigma^{\varphi}_{-i/2}(u_\iota) \xrightarrow{\;\; \iota\;\,}1_M$ in the $s^*$-topology,
and $\psi$ is lower

\noindent semicontinuous in the $s^*$-topology, we get
\medskip

\centerline{$\displaystyle
\psi (y^*y)\leq \varliminf\limits_{\iota} \psi (u_{\iota}y^*yu_{\iota}) =\varliminf\limits_{\iota}
\| J_{\varphi} \pi_{\varphi}(a^{1/2}) \pi_{\varphi}\big(\sigma^{\varphi}_{- i/2}(u_{\iota})\big)
J_{\varphi} y_{\varphi}\|^2$}

\noindent\hspace{2.53 cm}$\displaystyle
=\| J_{\varphi}  \pi_{\varphi}(a^{1/2}) J_{\varphi} y_{\varphi}\|^2\;\! .$
\medskip

\noindent Applying now \cite{Z2}, Corollary 1.2, we conclude :
\begin{equation}\label{ineq-phi}
\psi (y^*y)\leq \| (ya^{1/2})_{\varphi}\|^2 =\varphi (a^{1/2}y^*ya^{1/2})
=\varphi_a (y^*y)\;\! .
\end{equation}

In order to have (\ref{ineq-phi_a}) proved, we should show that (\ref{ineq-phi}) actually holds for every
$y\in\mathfrak{N}_{\varphi_a}\;\!$. This follows by repeating word for word the
corresponding part of the proof of \cite{Z2}, Lemma 2.1. We report it for sake of completeness.
\smallskip

For every $y\in\mathfrak{N}_{\varphi}\;\!$, since $s(a)\in M^{\varphi}$ and
$\mathfrak{N}_{\varphi} M^{\varphi}\subset \mathfrak{N}_{\varphi}\;\!$,
(\ref{ineq-phi}) yields
\smallskip

\centerline{$\psi \big( (1_M-s(a))y^*y(1_M-s(a))\big)\leq
\varphi \big( a^{1/2}(1_M-s(a))y^*y(1_M-s(a))a^{1/2}\big) =0\;\! .$}
\smallskip

\noindent $\mathfrak{N}_{\varphi}$ being $w$-dense in $M$, it follows
$\psi \big( 1_M-s(a)\big)=0\;\!$, what means $s(\psi )\leq s(a)\;\!$.
\smallskip

For $n\geq 1$ we consider the projection $e_n=\chi_{[1/n\;\! ,+\infty)}(a)\in
M^{\varphi}$, where $\chi_{\substack{ {} \\ [1/n\;\! ,+\infty)}}$ stands for the characteristic
function of $[1/n\;\! ,+\infty )\;\!$. Then $e_n \nearrow s(a)\;\!$.
We consider also the inverse $b_n$ of $a^{1/2}e_n$ in the reduced algebra
$e_nM^{\varphi}e_n\;\!$: $b_n=f_n(a)\in M^{\varphi}$ with
$\displaystyle f_n(t)=\frac 1{\;\! \sqrt{t}\;\!}\;\! \chi_{\substack{ {} \\ [1/n\;\! ,+\infty)}}(t)\;\!$.

Now let $y\in\mathfrak{N}_{\varphi_a}$ be arbitrary. Then $y\;\! a^{1/2}\in\mathfrak{N}_{\varphi}\;\!$, so
\medskip

\centerline{$y\;\! e_n =(y\;\! a^{1/2})\;\! b_n\in \mathfrak{N}_{\varphi}M^{\varphi}\subset
\mathfrak{N}_{\varphi}\;\! ,\qquad n\geq 1\;\! .$}
\medskip

\noindent Applying (\ref{ineq-phi}) and \cite{Z2}, Corollary 1.2, we obtain for every $n\geq 1$
\begin{equation*}
\begin{split}
\psi (e_ny^*y\;\! e_n)\leq\;& \varphi (a^{1/2}e_ny^*y\;\! e_na^{1/2}) =\| (y\;\! e_na^{1/2})_{\varphi}\|^2
=\| (y\;\! a^{1/2}e_n)_{\varphi}\|^2 \\
=\;&\| J_{\varphi} \pi_{\varphi}(e_n) J_{\varphi} (y\;\! a^{1/2})_{\varphi} \|^2
\end{split}
\end{equation*}
Since $s(\psi )\leq s(a)\;\! , e_n \nearrow s(a)$ and $\psi$ is lower semicontinuous in the
$s^*$-topology, it follows
\begin{equation*}
\begin{split}
\psi (y^*y)=\psi \big( s(a)\;\! y^*y\;\! s(a)\big) \leq\;& \varliminf\limits_{n\to\infty}\psi (e_ny^*y\;\! e_n) \\
\leq\;& \varliminf\limits_{n\to\infty} \| J_{\varphi} \pi_{\varphi}(e_n) J_{\varphi} (y\;\! a^{1/2})_{\varphi} \|^2 \\
=\;&  \| J_{\varphi} \pi_{\varphi}\big( s(a)\big) J_{\varphi} (y\;\! a^{1/2})_{\varphi} \|^2\;\! .
\end{split}
\end{equation*}
Applying \cite{Z2}, Corollary 1.2 again, we conclude :
\begin{equation*}
\psi (y^*y)\leq \|  \big( y\;\! a^{1/2}s(a)\big)_{\varphi} \|^2=\| (y\;\! a^{1/2})_{\varphi} \|^2
=\varphi (a^{1/2}y^*y\;\! a^{1/2}) =\varphi_a(y^*y)\;\! .
\end{equation*}

To complete the proof of the theorem, we have to find a $\sigma^{\varphi}$-invariant, hereditary
${}^*$-subalgebra $\mathfrak{M}_0$ of $\mathfrak{M}_{{\varphi}_a}$ such that
\medskip

\noindent\hspace{3.8 cm}$\mathfrak{M}\cap M^+\subset \mathfrak{M}_0\cap M^+\;\! ,$
\smallskip

\centerline{$\psi (b)={\varphi}_a (b)\;\! ,\qquad b\in \mathfrak{M}_0\cap M^+\;\! .$}
\smallskip

For we first notice :
\begin{itemize}
\item[(i)] $\{ b\in \mathfrak{M}_{\varphi_a}\cap M^+ ; \psi (b)=\varphi_a(b)\}\subset M$ is a face.
\item[(ii)] $\psi (b)={\varphi}_a (b)$ for all $b\in\mathfrak{M}\cap M^+$.
\end{itemize}

Since $\{ b\in \mathfrak{M}_{\varphi_a}\cap M^+ ; \psi (b)=\varphi_a(b)\}$ is clearly a convex cone,
for (i) we have only to verify the implication
\begin{equation}\label{eq-face}
M^+\ni b\leq c\in \mathfrak{M}_{\varphi_a}\cap M^+,\;\! \psi (c)=\varphi_a(c)
\Longrightarrow \psi (b) =\varphi_a(b)\;\! .
\end{equation}
It follows immediately by using
\medskip

\centerline{$\displaystyle \begin{array}{c}
\psi (b)\leq \varphi_a(b)\;\! ,\quad \psi (c-b)\leq \varphi_a(c-b)\;\! , \\
\psi (b)+\psi (c-b)=\psi (c)=\varphi_a(c)=\varphi_a(b)+\varphi_a(c-b)<+\infty\;\! .
\end{array}$}
\medskip

For (ii) let $b\in\mathfrak{M}\cap M^+$ be arbitrary. Without loss of generality we can assume
that $\| b\|\leq 1\;\!$.
Denoting $b_n:=1_M-(1_M-b)^n \in\mathfrak{M}\cap M^+,n\geq 1\;\!\;\! ,$ we obtain an

\noindent increasing sequence $(b_n)_{n\geq 1}$ which is $s^*$-convergent to the support $s(b)$
of $b$ (see e.g.

\noindent \cite{SZ1}, Section 2.22). Since all $b_n$ belong to the commutative $C^*$-subalgebra
of $M$

\noindent generated by $b\;\!$, the sequence $(b_nb\;\! b_n)_{n\geq 1}$ is still increasing and it is
$s^*$-convergent to $b\;\!$. Now we deduce successively :
\begin{itemize}
\item $\psi (b_n b_n)=\varphi_a(b_n b_n)$ for all $n\geq 1$ by the assumption on $\mathfrak{M}\;\!$;
\item $\psi (b_nb\;\! b_n)=\varphi_a(b_nb\;\! b_n)$ for all $n\geq 1$ by applying (\ref{eq-face}) with
$b=b_nb\;\! b_n$ and $c=b_nb_n\;\!$;
\item $\displaystyle \psi (b)=\lim\limits_{n\to\infty} \psi (b_nb\;\! b_n)=
\lim\limits_{n\to\infty} \varphi_a(b_nb\;\! b_n) =\varphi_a(b)$ by the normality of $\psi$ and $\varphi_a\;\!$.
\end{itemize}

Now we set
\smallskip

\centerline{$\displaystyle \mathfrak{F}_0:=\big\{ b\in \mathfrak{M}_{{\varphi}_a}\;\! ; 0\leq b
\leq c\text{ for some }c\in \mathfrak{M}\cap M^+\big\}\;\! .$}
\smallskip

\noindent\hspace{2.31 cm}$\mathfrak{N}_0:=\{ x\in M\;\! ; x^*x\in \mathfrak{F}_0\}\;\! ,$
\smallskip

\noindent\hspace{2.24 cm}$\mathfrak{M}_0:=\text{ linear span of }\mathfrak{N}_0^*\mathfrak{N}_0\;\! .$
\medskip

\noindent Then $\mathfrak{F}_0$ is a face, $\mathfrak{M}_0$ is a ${}^*$-subalgebra of $M$,
$\mathfrak{M}_0\cap M^+=\mathfrak{F}_0\;\!$, and $\mathfrak{M}_0$ is the linear span of
$\mathfrak{F}_0$ (see e.g. \cite{SZ1}, Proposition 3.21).
Thus $\mathfrak{M}_0$ is a hereditary ${}^*$-subalgebra of $\mathfrak{M}_{{\varphi}_a}$ and
$\mathfrak{M}\cap M^+\subset \mathfrak{F}_0=\mathfrak{M}_0\cap M^+\;\!$.
Since $\mathfrak{M}\cap M^+$ is $\sigma^{\varphi}$-invariant, also $\mathfrak{F}_0\;\!$,

\noindent and therefore $\mathfrak{M}_0$ is $\sigma^{\varphi}$-invariant.
Finally, the above (ii) and (i) imply that we have $\psi (b)={\varphi}_a (b)$ for all
$b\in \mathfrak{F}_0\;\!$.

\end{proof}

\begin{remark}
If $a$ is assumed only affiliated to $M^{\varphi}$ and not necessarily bounded, the

\noindent statement of Theorem \ref{ineq} is not more true. Counterexamples can be obtained using
\cite{PT}, Proposition 7.8 or \cite{Co2}, Exemple 8.

In both papers two faithful, semi-finite, normal weights $\psi_0\;\! ,\psi$ are constructed
on $B(\ell^2)$ such that
\smallskip

\centerline{$\psi_0\leq \psi\text{ and }\psi_0\neq \psi\;\! , \text{but }\psi_0(x)=\psi (x)\text{ for }x\in
\mathfrak{M}\cap M^+\;\! ,$}
\smallskip

\noindent where $\mathfrak{M}$ is a $w$-dense ${}^*$-subalgebra of ${\mathfrak{M}}_{\psi}$
(in \cite{Co2}, Exemple 8, the construction

\noindent delivers $\mathfrak{M}={\mathfrak{M}}_{\psi}$).

Now let $\varphi$ be a faithful, semi-finite, normal trace on $B(\ell^2)\;\!$.
By \cite{PT}, Theorem

\noindent 5.12 there exists a positive, self-adjoint operator $A$ on $\ell^2$,
necessarily affiliated to

\noindent $B(\ell^2)^{\varphi}=B(\ell^2)\;\!$, such that $\psi_0=\varphi_{\! \substack{ {} \\ A}}\;\!$.
Then
\begin{itemize}
\item $\varphi$ is a faithful, semi-finite, normal trace on $M=B(\ell^2)\;\!$,
\item $A$ is a positive, self-adjoint operator affiliated to $M^{\varphi}=B(\ell^2)\;\!$,
\item $\psi$ is a $\sigma^{\varphi}$-invariant, faithful, semi-finite, normal weight on $M$,
\item $\psi (x^*x)=\varphi_{\! \substack{ {} \\ A}}(x^*x)$ for $x\in\mathfrak{M}\;\!$, where $\mathfrak{M}$
is a $w$-dense ${}^*$-subalgebra of

\noindent $\mathfrak{M}_{\psi}\subset \mathfrak{M}_{{\psi}_0}=\mathfrak{M}_{\varphi_{\! \substack{ {} \\ A}}}\;\!$,
\end{itemize}
but $\psi \nleq \varphi_{\! \substack{ {} \\ A}}\;\!$, because otherwise it would follow $\psi\leq
\varphi_{\! \substack{ {} \\ A}} =\psi_0\;\!$, hence $\psi =\psi_0\;\!$, in

\noindent contradiction to $\psi\neq\psi_0\;\!$.

\end{remark}

\begin{remark}
If in Theorem \ref{ineq} it is additionally assumed that $1_M-s(\psi )$ belongs to the $w$-closure of
$\{ y\in \mathfrak{M}_{\varphi}\;\! ; y\;\! s(\psi )=0\}$ (that happens, for example, if $s(\psi )\in M^{\varphi}_{\infty}\;\!$,
because $\mathfrak{M}_{\varphi} \mathfrak{M}^{\varphi}_{\infty}\subset\mathfrak{M}_{\varphi}$), then it
follows also the equality $s(\psi )=s(a)\;\!$.

Since $s(\psi )\leq s(a)$ trivially, we have to verify that for any $y\in\mathfrak{M}_{\varphi}$ with
$y\;\! s(\psi )=0\;\!$, that is with $\psi (y^*y)=0\;\!$, we have $y\;\! s(a)=0\;\!$. We can argue as in
the proof of \cite{PT}, Lemma 5.2 :

By (\ref{conv.u}), by the lower semicontinuity of ${\varphi}_a$ in the $s^*$-topology, and by (\ref{left-right-u}),
we have

\centerline{$\displaystyle {\varphi}_a(y^*y)\leq \varliminf\limits_{\iota} {\varphi}_a (u_{\iota}y^*yu_{\iota}) =
\varliminf\limits_{\iota}\psi (u_{\iota}y^*yu_{\iota})\;\! .$}
\smallskip

\noindent Using now the inequalities
\begin{equation*}
\begin{split}
u_{\iota}y^*yu_{\iota}\leq\;& \big( 2\cdot 1_M-u_{\iota}\big) y^*y \big( 2\cdot 1_M-u_{\iota}\big) +
u_{\iota}y^*yu_{\iota} \\
=\;&2 \Big( (1_M-u_{\iota}) y^*y (1_M-u_{\iota}) +y^*y\Big)\;\! ,
\end{split}
\end{equation*}
and $\psi\leq {\varphi}_a\;\!$, as well as \cite{Co2}, Lemme 7 (b) or \cite{Z2}, Proposition 1.1, we obtain
\begin{equation*}
\begin{split}
{\varphi}_a(y^*y)\leq\;& 2 \varliminf\limits_{\iota}\psi \Big( (1_M-u_{\iota}) y^*y (1_M-u_{\iota}) \Big) \\
\leq\;& 2 \varliminf\limits_{\iota} {\varphi}_a \Big( (1_M-u_{\iota}) y^*y (1_M-u_{\iota}) \Big)
=2 \varliminf\limits_{\iota}\big\| \big(y (1_M-u_{\iota}) a^{1/2}\big)_{\varphi}\big\|^2 \\
=\;&2 \varliminf\limits_{\iota}
\big\| J_{\varphi} \pi_{\varphi}\big( \sigma^{\varphi}_{- i/2}\big( a^{1/2}(1_M-u_{\iota})\big)\big)
J_{\varphi} y_{\varphi}\big\|^2 \\
=\;&2 \varliminf\limits_{\iota}
\big\| J_{\varphi} \pi_{\varphi}(a^{1/2}) \pi_{\varphi}\big( 1_M-\sigma^{\varphi}_{- i/2}(u_{\iota})\big)
J_{\varphi} y_{\varphi} \big\|^2
\end{split}
\end{equation*}
Since, by (\ref{conv.u}), $\sigma^{\varphi}_{- i/2}(u_{\iota}) \xrightarrow{\;\; \iota\;\,}1_M$ in the
$s^*$-topology, we conclude that ${\varphi}_a(y^*y)=0\;\!$, what is equivalent to $y\;\! a^{1/2}=0
\Longleftrightarrow y\;\! s(a)=0\;\!$.
\end{remark}

The next theorem is a slight extension of \cite{Z2}, Theorem 2.3 :

\begin{theorem}\label{eq-1}
Let $M$ be a $W^*$-algebra, $\varphi$ a faithful, semi-finite, normal weight on $M$, $a\in (M^{\varphi})^+$,
and $\psi$ a $\sigma^{\varphi}$-invariant, normal weight on $M$. If there exists a $w$-dense,
$\sigma^{\varphi}$-invariant ${}^*$-subalgebra $\mathfrak{M}$ of $\mathfrak{M}_{{\varphi}_a}$ such that
\medskip

\centerline{$\psi (x^*x)={\varphi}_a(x^*x)\;\! ,\qquad x\in\mathfrak{M}\;\! ,$}
\smallskip

\noindent then

\centerline{$\psi ={\varphi}_a\;\! .$}
\end{theorem}

\begin{proof}
By Theorem \ref{ineq} we have $\psi\leq\varphi_a\;\!$. In particular, $\psi$ is semi-finite.

On the other hand, by \cite{PT}, Theorem 5.12 there exists a positive, self-adjoint operator
$A\;\!$, affiliated to $M^{\varphi}$, such that $\psi =\varphi_{\! \substack{ {} \\ A}}\;\!$.
Since $\varphi_{\! \substack{ {} \\ A}} =\psi\leq\varphi_a\;\!$,
Lemma \cite{Z2}, 2.2 yields $A\leq a\;\!$. In particular, $A$ is bounded.

Since $\mathfrak{M}_{\varphi_{\! \substack{ {} \\ A}}}$ is the linear span of
$\{ b\in M^+ : \varphi_{\! \substack{ {} \\ A}}(b)<+\infty\}\;\!$, $\mathfrak{M}_{\varphi_a}$ is
the linear span of $\{ b\in M^+ : \varphi_a(b)<+\infty\}\;\!$, and $\varphi_{\! \substack{ {} \\ A}}
\leq\varphi_a\;\!$, we have $\mathfrak{M}\subset \mathfrak{M}_{\varphi_a}\subset
\mathfrak{M}_{\varphi_{\! \substack{ {} \\ A}}}\;\!$. Therefore we can apply Theorem \ref{ineq}
again and deduce that $\varphi_a\leq \varphi_{\! \substack{ {} \\ A}}=\psi\;\!$.

\end{proof}

An equivalent, but slightly more symmetric form of Theorem \ref{eq-1} is

\begin{theorem}\label{eq-2}
Let $M$ be a $W^*$-algebra, $\varphi$ a faithful, semi-finite, normal weight on $M$, $a\;\! ,b\in (M^{\varphi})^+$,
and $\psi$ a $\sigma^{\varphi}$-invariant, normal weight on $M$. If there exists a $w$-dense,
$\sigma^{\varphi}$-invariant ${}^*$-subalgebra $\mathfrak{M}$ of $\mathfrak{M}_{{\varphi}_a}$ such that
\medskip

\centerline{$\psi_b (x^*x)={\varphi}_a(x^*x)\;\! ,\qquad x\in\mathfrak{M}\;\! ,$}
\smallskip

\noindent then

\centerline{$\psi_b ={\varphi}_a\;\! .$}
\end{theorem}

\begin{proof}
Since $\psi$ is $\sigma^{\varphi}$-invariant and $b\in (M^{\varphi})^+$, the normal weight $\psi_b$ is still

\noindent $\sigma^{\varphi}$-invariant : we have for every $t\in\mathbb{R}$ and every $x\in M^+$
\medskip

\centerline{$\psi_b\big( \sigma^{\varphi}_t(x)\big) =\psi\big( b^{1/2}\sigma^{\varphi}_t(x)b^{1/2}\big)
=\psi \big( \sigma^{\varphi}_t( b^{1/2}xb^{1/2})\big) =\psi \big( b^{1/2}xb^{1/2}\big) =\psi_b (x)\;\! .$}
\smallskip

\noindent Thus we can apply Theorem \ref{eq-1} with $\psi$ replaced by $\psi_b\;\!$.

\end{proof}

An immediate consequence of Theorems \ref{eq-1} and \ref{eq-2} is \cite{PT}, Proposition 5.9 :

\begin{corollary}\label{PedTak}
Let $M$ be a $W^*$-algebra, $\varphi$ a faithful, semi-finite, normal weight on $M$,
and $\psi$ a $\sigma^{\varphi}$-invariant, normal weight on $M$. If there exists a $w$-dense,

\noindent $\sigma^{\varphi}$-invariant ${}^*$-subalgebra $\mathfrak{M}$ of $\mathfrak{M}_{\varphi}$
such that
\medskip

\centerline{$\psi (x^*x)={\varphi}(x^*x)\;\! ,\qquad x\in\mathfrak{M}\;\! ,$}
\smallskip

\noindent then

\centerline{$\psi =\varphi\;\! .$}
\end{corollary}

\hfill $\square$
\smallskip

The next result is a balanced counterpart of Theorem \ref{eq-2} :

\begin{theorem}\label{eq-3}
Let $M$ be a $W^*$-algebra, $\varphi\;\! ,\psi$ faithful, semi-finite, normal weights on $M$,
and $a \in (M^{\varphi})^+\;\! ,b\in (M^{\psi})^+$.
Assume that there are a $w$-dense, $\sigma^{\varphi}$-invariant ${}^*$-subalgebra $\mathfrak{M}_1$
of $\mathfrak{M}_{{\varphi}_a}$ and a $w$-dense, $\sigma^{\varphi}$-invariant ${}^*$-subalgebra
$\mathfrak{M}_2$ of $\mathfrak{M}_{{\psi}_b}$ such that
\smallskip

\centerline{$\psi_b (x^*x)={\varphi}_a(x^*x)\;\! ,\qquad x\in\mathfrak{M}_1\cup\mathfrak{M}_2\;\! .$}
\smallskip

\noindent Then

\centerline{$\psi_b ={\varphi}_a\;\! .$}
\end{theorem}

\begin{proof}
We have just to apply twice Theorem \ref{ineq}.

\end{proof}

Immediate consequences of Theorem \ref{eq-3} are :

\begin{theorem}\label{eq-4}
Let $M$ be a $W^*$-algebra, $\varphi\;\! ,\psi$ faithful, semi-finite, normal weights on $M$,
and $a \in (M^{\varphi})^+\;\! ,b\in (M^{\psi})^+$.
Assume that there exists a $w$-dense, both $\sigma^{\varphi}$- and $\sigma^{\psi}$-invariant 
${}^*$-subalgebra $\mathfrak{M}$ of $\mathfrak{M}_{{\varphi}_a}\cap\mathfrak{M}_{{\psi}_b}$
such that
\smallskip

\centerline{$\psi_b (x^*x)={\varphi}_a(x^*x)\;\! ,\qquad x\in\mathfrak{M}\;\! ,$}
\smallskip

\noindent Then

\centerline{$\psi_b ={\varphi}_a\;\! .$}
\end{theorem}

\hfill $\square$
\smallskip

\begin{corollary}\label{PedTak-symm}
Let $M$ be a $W^*$-algebra and $\varphi\;\! ,\psi$ faithful, semi-finite, normal weights on $M$.
If there exists a $w$-dense, both $\sigma^{\varphi}$- and $\sigma^{\psi}$-invariant 
${}^*$-subalgebra

\noindent $\mathfrak{M}$ of $\mathfrak{M}_{{\varphi}_a}\cap\mathfrak{M}_{{\psi}_b}$
such that
\smallskip

\centerline{$\psi (x^*x)={\varphi}(x^*x)\;\! ,\qquad x\in\mathfrak{M}\;\! ,$}
\smallskip

\noindent then

\centerline{$\psi =\varphi\;\! .$}
\end{corollary}

\hfill $\square$
\smallskip

There exist also criteria of different kind for equality and inequalities between faithful,
semi-finite, normal weights, due to A. Connes. They are in terms of the {\it Connes cocycle}
(see \cite{Co1}, Section 1.2) or \cite{SZ1}, Theorem 10.28 and C.10.4) : if $\varphi$
and $\psi$ are faithful, semi-finite, normal weights on a $W^*$-algebra, the Connes
cosycle of $\psi$ with respect to $\varphi$ will be denoted by $(D\psi : D\varphi )_t\;\! ,
t\in \mathbb{R}\;\!$.

The next theorem frames up results from \cite{Co1} and \cite{Co2}. It will be used in the next
section.

\begin{theorem}\label{cocycle}
Let $M$ be a $W^*$-algebra, and $\varphi\;\! ,\psi$ faithful, semi-finite, normal weights on $M$.
\begin{itemize}
\item[(i)] $\psi =\varphi$ if and only if $(D\psi : D\varphi )_t=1_M$ for all $t\in\mathbb{R}\;\!$.
\item[(ii)] $\psi \leq\varphi$ if and only if $\mathbb{R}\ni t\longmapsto (D\psi : D\varphi )_t\in M$
allows a $w$-continuous

\noindent extension $\{\zeta\in\mathbb{C}\;\! ; -\frac 12\leq {\rm Im} \zeta\leq 1\}\ni \zeta\longmapsto
(D\psi : D\varphi )_{\zeta}\in M\;\!$, which is

\noindent analytic in the interior and satisfies $\| (D\psi : D\varphi )_{-\frac i2}\|\leq 1\;\!$.
\item[(iii)] $\psi (a)=\varphi (a)$ for all $a\in\mathfrak{F}_{\varphi}$ if and only if
$\mathbb{R}\ni t\longmapsto (D\psi : D\varphi )_t\in M$ has a $w$-continuous extension
$\{\zeta\in\mathbb{C}\;\! ; -\frac 12\leq {\rm Im} \zeta\leq 1\}\ni \zeta\longmapsto
(D\psi : D\varphi )_{\zeta}\in M\;\!$, which is analytic in the interior and such that
$(D\psi : D\varphi )_{-\frac i2}$ is isometric.
\end{itemize}
\end{theorem}

\begin{proof}
(i) is consequence of \cite{Co1}, Lemme 1.2.3 (b) (see also \cite{St}, Corollary 3.6), (ii) is
\cite{Co2}, Th\'eor\`eme 3 (see also \cite{St}, Corollary 3.13), and (iii) is \cite{Co2}, Th\'eor\`eme 4
(see

\noindent also \cite{St}, Corollary 3.14).

\end{proof}
\smallskip

\section{Inequalities and equality between operator valued weights}

\cite{H1}, Proposition 2.3, \cite{H1}, Lemma 4.8 and \cite{H2},  Remark on page 360
indicate a way to reduce investigation of equality and inequalities between faithful,
semi-finite, normal operator valued weights to verification of equality and inequalities
between usual, scalar valued faithful, semi-finite, normal weight.
Actually, most part of the above results holds, with essentially the same
proof, also for not necessarily faithful weights and operator valued weights.
Let us single out those statements which will be used in the sequel :

\begin{proposition}\label{H-revisited}
Let $M$ be a $W^*$-algebra, $1_M\in N\subset M$ a $W^*$-subalgebra, and

\noindent $E\;\! ,E_1\;\! ,E_2 : M^+\longrightarrow \overline{N}\;\!^+$ semi-finite, normal
operator valued weights.
\begin{itemize}
\item [(i)] If $\varphi$ is a semi-finite, normal weight on $N$, then $\varphi\circ E$ is a
normal weight on $M$ such that $\mathfrak{M}_E\cap\mathfrak{M}_{\varphi\circ E}$
is a $w$-dense ${}^*$-subalgebra of $M$. In particular,

\noindent $\varphi\circ E$ is a semi-finite, normal weight on $M$. Assuming additionally
that

\noindent $E$ and $\varphi$ are faithful, also $\varphi\circ E$ is faithful.
\item [(ii)] If $\varphi\circ E_2\leq \varphi\circ E_1$ for every faithful, semi-finite, normal
weight $\varphi$ on $N$,

\noindent then $E_2\leq E_1\;\!$.
\end{itemize}
\end{proposition}

\begin{proof}
For (i) we have essentially to repeat the proof of \cite{H1}, Proposition 2.3 :

Since $\varphi$ is semi-finite, there exists a bounded net
$(x_{\iota})_{\iota}$ in $\mathfrak{M}_{\varphi}$ such that $x_{\iota}  \xrightarrow{\;\; \iota\;\,}1_M$
in the $s^*$-topology.
For every $a\in\mathfrak{M}_{E}\cap M^+$ and every $\iota\;\!$, we have
\begin{equation*}
E(x_{\iota}^*a\;\! x_{\iota})=x_{\iota}^*E(a)\;\! x_{\iota}\in N\;\! ,
\end{equation*}
so $x_{\iota}^*a\;\! x_{\iota}\in \mathfrak{M}_{E}\cap M^+$. On the other hand,
\begin{equation*}
(\varphi\circ E)(x_{\iota}^*a\;\! x_{\iota})=\varphi\big( x_{\iota}^*E(a)\;\! x_{\iota}\big)
\leq \| E(a)\| \varphi (x_{\iota}^*x_{\iota})<+\infty\;\! ,
\end{equation*}
yields $x_{\iota}^*a\;\! x_{\iota}\in \mathfrak{M}_{\varphi\circ E}\cap M^+\!$, hence
$x_{\iota}^*a\;\! x_{\iota}\in  \mathfrak{M}_{E}\cap \mathfrak{M}_{\varphi\circ E}\;\!$.
Since $x_{\iota}^*a\;\! x_{\iota}  \xrightarrow{\;\; \iota\;\,} a$ in the $s^*$-topology,
it follows that $a$ belongs to the $s^*$-closure of $ \mathfrak{M}_{E}\cap
\mathfrak{M}_{\varphi\circ E}\;\!$.

On the other hand, the faithfulness of $E$ and $\varphi$ implies the faithfulness
of $\varphi\circ E$ trivially.
\smallskip

For the proof of (ii) we use the argumentation of the proof of \cite{H1}, Lemma 4.8 :

Let $\psi$ be any normal positive form on $N$. Choosing some semi-finite, normal
weight $\theta$ on $N$ with support $s(\theta )=1_M-s(\psi )\;\!$, $\varphi :=\psi +\theta$
will be a faithful. semi-finite, normal weight on $N$ such that $\psi (b)=\varphi \big( s(\psi )\;\!
b\;\! s(\psi )\big)$ for all $b\in N^+$,

\noindent and consequently, for all $b\in \overline{N}\;\!^+$.
By the assumption on $E_1$ and $E_2\;\!$, we have

\noindent $\varphi\circ E_2\leq \varphi\circ E_1$
and it follows for each $a\in M^+$ :
\begin{equation*}
\begin{split}
\psi\big( E_2(a)\big) =\;&\varphi\big( s(\psi )\;\! E_2(a)\;\! s(\psi )\big) =(\varphi\circ E_2)\big( s(\psi )\;\!
a\;\! s(\psi )\big) \\
\leq\;&(\varphi\circ E_1)\big( s(\psi )\;\!a\;\! s(\psi )\big) =\varphi\big( s(\psi )\;\! E_1(a)\;\! s(\psi )\big)
=\psi\big( E_1(a)\big)\;\! .
\end{split}
\end{equation*}

In conclusion, for every $a\in M^+$ holds
\smallskip

\centerline{$\psi\big( E_2(a)\big)\leq \psi\big( E_1(a)\big)$ for all normal positive forms $\psi$ on $N ,$}
\smallskip

\noindent what means the inequality $E_2(a)\leq E_1(a)$ in $\overline{N}\;\!^+$.
In other words, $E_2\leq E_1\;\!$.

\end{proof}

For  two weights $\varphi\;\! ,\psi$ on a $W^*$-algebra $M$ holds trivially :
If $\mathfrak{M}_{\varphi}\subset\mathfrak{M}_{\psi}$ and the restrictions of $\varphi$
and $\psi$ to $\mathfrak{M}_{\varphi}$ coincide, then $\psi\leq \varphi\;\! .\!$
Indeed, at any element of $M^+$ which does not belong to $\mathfrak{M}_{\varphi}\;\!$,
$\varphi$ takes the value $+\infty$ which majorizes any extended real valiue.

Here we will prove that the analogous statement holds true also for semi-finite, normal
operator valued weights on $W^*$-algebras. The proof will use modular theory, it is
not such trivial as in the case of scalar weights.

Let us first examine the situation $E_2 \restriction (\mathfrak{M}_{E_1}\cap M^+) =
E_1 \restriction (\mathfrak{M}_{E_1}\cap M^+)$ for semi-finite, normal operator valued
weights $E_1\;\! ,E_2 : M\longrightarrow \overline{N}\;\!^+$, $M$ a $W^*$-algebra and
$1_M\in N\subset M$ a $W^*$-subalgebra, in the case of faithful $E_1\;\!$:

\begin{lemma}\label{faithful-ineq}
Let $M$ be a $W^*$-algebra, and $1_M\in N\subset M$ a $W^*$-subalgebra. If

\noindent $E_1\;\! ,E_2 : M^+\longrightarrow \overline{N}\;\!^+$ are semi-finite,
normal operator valued weights, $E_1$ faithful, such that
\smallskip

\centerline{$\mathfrak{M}_{E_1}\subset\mathfrak{M}_{E_2}\text{ and }\;\! E_2(a)=E_1(a)\text{ for
all }a\in \mathfrak{M}_{E_1}\cap M^+ ,$}
\medskip

\noindent then $E_2\leq E_1\;\! .\!$
Moreover, for each faithful, semi-finite, normal weight $\varphi$ in $N$,
\medskip

\centerline{$\mathfrak{M}_{\varphi\circ E_1}\subset \mathfrak{M}_{\varphi\circ E_2}\text{ and }\;\!
(\varphi\circ E_2)(a)=(\varphi\circ E_1)(a)\text{ for all }a\in \mathfrak{M}_{\varphi\circ E_1}\cap M^+ .$}
\end{lemma}

\begin{proof}
Let $\varphi$ be an arbitrary faithful, semi-finite, normal weight on $N$.

By Proposition \ref{H-revisited} (i) and by the assumptions of the lemma, $\varphi\circ E_1$ and $\varphi\circ E_2$ are semi-finite, normal
weights on $M$, $\varphi\circ E_1$ is faithful, and $\mathfrak{M}:=\mathfrak{M}_{E_1}\cap\mathfrak{M}_{\varphi\circ E_1}$ is a $w$-dense ${}^*$-subalgebra of $\mathfrak{M}_{\varphi\circ E_1}$ satisfying
\begin{equation*}
(\varphi\circ E_2) (x^*x)=(\varphi\circ E_1) (x^*x)\;\! ,\qquad x\in\mathfrak{M}\;\! .
\end{equation*}
We can apply Theorem \ref{ineq} to deduce the inequality $\varphi\circ E_2\leq\varphi\circ E_1$
once we verify the $\sigma^{\varphi\circ E_1}$-invariance of $\mathfrak{M}\;\!$. But this is
consequence of the $\sigma^{\varphi\circ E_1}$-invariance of both  $\mathfrak{M}_{E_1}\!$
(see \cite{H1}, Proposition 4.9) and $\mathfrak{M}_{\varphi\circ E_1}\;\!$.

Now, having $\varphi\circ E_2\leq\varphi\circ E_1$ for any faithful, semi-finite, normal weight $\varphi$
on $N$, the inequality $E_1\leq E_2$ follows by applying Proposition \ref{H-revisited} (ii).
\smallskip

For the second statement of the lemma, let the faithful, semi-finite, normal weight $\varphi$ on
$N$ and $a\in M^+$ verifying $\varphi\big( E_1(a)\big) <+\infty$ be arbitrary. We have to show
that then
\begin{equation}\label{scalar-eq}
\varphi\big( E_2(a)\big) =\varphi\big( E_1(a)\big) .
\end{equation}

Let us consider $M$ in a spatial representation $M\subset B(H)\;\!$.

Since each normal positive form on a von Neumann algebra is a countable sum of positive vector
functionals (see e.g. \cite{Dix}, Chap. I,  \S 4, Th\'eor\`eme 1 or \cite{SZ1}, E.7.8)

\noindent and each normal
weight on a von Neumann algebra is a sum of normal positive forms (see \cite{PT}, Theorem 7.2),
$\varphi$ is a sum of positive vector functionals on $N$.

Since $E_2(a)\leq E_1(a)$ belong to $\overline{N}\;\!^+\!$ and
$\varphi\big( E_2(a)\big) \leq\varphi\big( E_1(a)\big) <+\infty$ with $\varphi$

\noindent faithful, by \cite{H1}, Lemma 1.4 they correspond to positive self-adjoint
operators $A_1$

\noindent and $A_2$ on $H$, affiliated with $N$ and such that, for every $\xi\in H$,
\medskip

\centerline{$\displaystyle
\omega_{\xi}\big( E_j(a) \big) =
\begin{cases}
\; \| A_j^{1/2}\xi\|^2\!\! &\text{for $\xi\in \mathcal{D}(A_j^{1/2})$}\\
\;\;\,\quad \infty &\text{otherwise}
\end{cases}\;\! ,\qquad j=1\;\! ,2\;\! ,$}
\smallskip

\noindent where $\mathcal{D}(A_j^{1/2})$ stands for the domain of $A_j^{1/2}$.
Having $\omega_{\xi}\big( E_2(a) \big)\leq \omega_{\xi}\big( E_1(a) \big)$ for
all $\xi\in H$, it follows
\begin{equation*}
\mathcal{D}(A_1^{1/2})\subset \mathcal{D}(A_2^{1/2})\text{ and }\| A_2^{1/2}\xi\|
\leq \| A_1^{1/2}\xi\|\text{ for all }\xi\in \mathcal{D}(A_1^{1/2})\;\! .
\end{equation*}
Consequently the formulas
\medskip

\centerline{$U\big( A_1^{1/2}\xi\big) := A_2^{1/2}\xi$ for $\xi\in \mathcal{D}(A_1^{1/2})\;\! ,\quad
U\eta :=0$ for $\eta\in \big( 1_H-s(A_1)\big) H$}
\smallskip

\noindent define a linear contraction $U : H\longrightarrow H$ for which $U A_1^{1/2}
\subset A_2^{1/2}$.

Let us consider, for any $n\geq 1\;\!$, the spectral projection
$f_n:=\chi_{\substack{ {} \\ {[1/n,n]}}}(A_1)\in N$ of

\noindent $A_1\;\!$. We notice that $f_n\nearrow s(A_1)\;\!$.

The boundedness of the operator
$A_1f_n$ yields $E_1(f_n a\;\! f_n)=
f_n E_1(a)\;\! f_n\in N\;\!$, hence $f_n a\;\! f_n\in \mathfrak{M}_{E_1}\cap M^+\subset
\mathfrak{M}_{E_2}\cap M^+$ and $E_2(f_n a\;\! f_n)=E_1(f_n a\;\! f_n)\;\!$.
For every $\xi\in H$ we deduce
\begin{equation*}
\begin{split}
\| A_1^{1/2} f_n\xi\|^2=\;&\omega_{f_n\xi} \big( E_1(a)\big)\! =\omega_{\xi}\big( f_n E_1(a)\;\! f_n\big)\!
=\omega_{\xi}\big( E_1(f_n a\;\! f_n)\big) \\
=\;&\omega_{\xi}\big( E_2(f_n a\;\! f_n)\big)\! =\omega_{\xi}\big( f_n E_2(a)\;\! f_n\big)\!
=\omega_{f_n\xi}\big( E_2(a)\big)\! \\
=\;&\| A_2^{1/2} f_n\xi\|^2 =\| U A_1^{1/2} f_n\xi\|^2\;\! .
\end{split}
\end{equation*}
Since $ A_1^{1/2} f_nH=f_nH$, it follows that $U$ acts isometrically on $f_nH$.

Taking into account that $\bigcup\limits_{n\geq 1}f_nH$ is dense in $s(A_1)H$, we conclude
that $U$ acts isometrically on $s(A_1)H$, hence $U$ is a partial isometry with
initial projection $s(A_1)\;\!$. Therefore, for every $\xi\in \mathcal{D}(A_1^{1/2})\;\!$,
\begin{equation}\label{eq.vect.funct}
\omega_{\xi}\big( E_1(a)\big) =\| A_1^{1/2}\xi\|^2=\| UA_1^{1/2}\xi\|^2=\| A_2^{1/2}\xi\|^2
=\omega_{\xi}\big( E_2(a)\big)\;\! .
\end{equation}

Let $(\xi_{\iota})_{\iota}$ be a family of vectors in $H$ such that $\varphi$ is equal to the
sum $\sum\limits_{\iota}\omega_{\xi_{\iota}}\;\!$.

\noindent Since $\sum\limits_{\iota} \omega_{\xi_{\iota}}\big( E_1(a)\big) =\varphi
\big( E_1(a)\big) <+\infty\;\!$, each $\xi_{\iota}$ should belong to $\mathcal{D}(A_1^{1/2})\;\!$.
Indeed, $\xi_{\iota}\notin \mathcal{D}(A_1^{1/2})$ for some $\iota$ would imply $+\infty =
\omega_{\xi_{\iota}}\big( E_1(a)\big)\leq \varphi \big( E_1(a)\big)$.
Thus (\ref{eq.vect.funct})

\noindent holds for each $\xi_{\iota}$ and (\ref{scalar-eq}) folows :
\begin{equation*}
\varphi \big( E_1(a)\big) =\sum\limits_{\iota} \omega_{\xi_{\iota}}\big( E_1(a)\big)
=\sum\limits_{\iota} \omega_{\xi_{\iota}}\big( E_2(a)\big) =\varphi \big( E_2(a)\big) .
\end{equation*}

\end{proof}

The next lemma is an inverse to Lemma \ref{faithful-ineq} :

\begin{lemma}\label{inv-faithful-ineq}
Let $M$ be a $W^*$-algebra, and $1_M\in N\subset M$ a $W^*$-subalgebra. If

\noindent $E_1\;\! ,E_2 : M^+\longrightarrow \overline{N}\;\!^+$ are semi-finite,
normal operator valued weights, $E_1$ faithful, such that, for each faithful, semi-finite,
normal weight $\varphi$ in $N$,
\medskip

\centerline{$\mathfrak{M}_{\varphi\circ E_1}\subset \mathfrak{M}_{\varphi\circ E_2}\text{ and }\;\!
(\varphi\circ E_2)(a)=(\varphi\circ E_1)(a)\text{ for all }a\in \mathfrak{M}_{\varphi\circ E_1}\cap M^+ ,$}
\smallskip

\noindent then

\centerline{$\mathfrak{M}_{E_1}\subset \mathfrak{M}_{E_2}\text{ and }\;\! E_2(a)=E_1(a)\text{ for
all }a\in \mathfrak{M}_{E_1}\cap M^+ .$}
\end{lemma}

\begin{proof}
First of all, by Proposition \ref{H-revisited} (ii) we have $E_2\leq E_1\;\!$, so
$\mathfrak{M}_{E_1}\subset \mathfrak{M}_{E_2}\;\!$.

Now let $a\in \mathfrak{M}_{E_1}\cap M^+$ be arbitrary. Then $E_2(a)\leq E_1(a)\in N$,
in particular also $E_2(a)\in N$. All we have now to show is the equality $E_2(a)=E_1(a)\;\!$.
\smallskip

Since $E_1(\mathfrak{M}_{E_1})$ is a $w$-dense two-sided ideal in $N$ (see \cite{H1},
Proposition 2.5), there exists a family $(f_{\iota})_{\iota\in I}$ of non-zero, countably decomposable
projections in $N$, belonging to $E_1(\mathfrak{M}_{E_1})\;\!$, such that
$\displaystyle \sum\limits_{\iota\in I} f_{\iota} =1_M\;\!$.

For let us consider, by the Zorn lemma, a maximal family $(f_{\iota})_{\iota\in I}$ of mutually

\noindent orthogonal, non-zero, countably decomposable projections in $N$, each one belonging
to $E_1(\mathfrak{M}_{E_1})\;\!$.

Assume that $\displaystyle f_0:=1_M-\sum\limits_{\iota\in I} f_{\iota}\neq 0\;\!$. 
Since $E_1(\mathfrak{M}_{E_1})$ is $w$-dense, there is some

\noindent $y\in E_1(\mathfrak{M}_{E_j})$ with $y f_0\neq 0\;\!$. Further, since
$E_1(\mathfrak{M}_{E_1})$ is a two-sided ideal in $N$,

\noindent $f_0y^*y f_0\neq 0$ belongs to $E_1(\mathfrak{M}_{E_1})$ and, choosing some
$0<\lambda <\| f_0y^*y f_0\|\;\!$, we have
$f_1:=\chi_{\substack{ {} \\ [\lambda\;\! ,+\infty )}}(f_0y^*y f_0)\neq 0\;\!$. Choose some
non-zero, countably decomposable

\noindent projection $f_2\leq f_1$ in $N$. Then $\displaystyle f_2\leq f_1\leq
\frac 1{\;\!\lambda\;\!}\;\! f_0y^*y f_0\in E_1(\mathfrak{M}_{E_1})\;\!$. Consequently

\noindent $0\neq f_2\leq f_0$ and, since $E_1(\mathfrak{M}_{E_1})$ is a two-sided ideal in $N$,
$f_2\in E_1(\mathfrak{M}_{E_1})\;\!$. But this contradicts the maximality of the family
$(f_{\iota})_{\iota\in I}\;\!$.
Consequently $f_0$ should vanish, that is $\displaystyle \sum\limits_{\iota\in I} f_{\iota} =1_M\;\!$.

Set $f_F:=\sum\limits_{\iota\in F} f_{\iota} $ for any finite subset $F\subset I$.
$f_F$ is then a countably decomposable projection belonging to $E_1(\mathfrak{M}_{E_1})\;\!$.
Since $f_F\nearrow 1_M\;\!$, for the equality  $E_2(a)=E_1(a)$ it is enough to show that
$f_F E_2(a)f_F =f_F E_1(a)f_F$ for every finite $F\subset I$.
\smallskip

For let $F$ be any finite subset of $I$. $f_F E_2(a)f_F =f_F E_1(a)f_F$ will follow if we

\noindent show that $\psi\big( f_F E_2(a)f_F\big) =\psi\big( f_F E_1(a)f_F\big)$ for every
normal positive form $\psi$ on $N$ with $s(\psi )=f_F\;\!$.

Choose for each $\iota\in I\setminus F$ a normal positive form $\psi_{\iota}$ on $N$ of
support $s(\psi_{\iota})=f_{\iota}$ and define the faithful, semi-finite, normal weight $\varphi$
on $N$ by taking $\varphi :=\psi +\sum\limits_{\iota\in I\setminus F}\psi_{\iota}$ on $N^+$.
Since
\smallskip

\centerline{$(\varphi\circ E_1) (f_F a f_F) =\varphi\big( f_F E_1(a) f_F\big) \leq \| E_1(a)\|\;\!
\varphi (f_F)=\| E_1(a)\|\;\! \psi (f_F)<+\infty\;\! ,$}
\smallskip

\noindent  $f_F a f_F$ belongs to $\mathfrak{M}_{E_1}\cap M^+$ and
by the assumption on $E_1$ and $E_2$ we deduce :
\begin{equation*}
\begin{split}
\psi\big( f_F E_2(a)f_F\big) =\;&\varphi \big( f_F E_2(a)f_F\big) =(\varphi\circ E_2) (f_F a f_F) \\
=\;&(\varphi\circ E_1) (f_F a f_F)=\varphi \big( f_F E_1(a)f_F\big) =\psi\big( f_F E_1(a)f_F\big)\;\! .
\end{split}
\end{equation*}

\end{proof}

The case of general semi-finite, normal operator valued weights will be reduced to
the above case using the next proposition :

\begin{proposition}\label{reduce}
Let $M$ be a $W^*$-algebra, $1_M\in N\subset M$ a $W^*$-subalgebra, and

\noindent $E_0 : M^+\longrightarrow \overline{N}\;\!^+$ a semi-finite, normal
operator valued weight. Let $p_0$ denote the

\noindent central support of $s(E_0)\in N'\cap M$ in $N$, and $\pi_{\substack{ {} \\ 0}}$ the
${}^*$-isomorphism
\smallskip

\centerline{$N p_0\ni y\longmapsto y s(E_0)\in N s(E_0)\;\! .$}
\smallskip

\noindent Then there exists an one-to-one correspondence between the semi-finite,
normal operator valued weights
\medskip

\centerline{$E : M^+\! \longrightarrow \overline{N}\;\!^+\!$ of support $\;\! s(E)\leq s(E_0)\;\! ,$}
\smallskip

\noindent and the semi-finite, normal operator valued weights
\medskip

\centerline{$\widetilde{E} : s(E_0)M^+\! s(E_0) \longrightarrow \overline{N s(E_0)}\;\!^+ ,$}
\smallskip

\noindent such that
\begin{equation*}
\begin{split}
\widetilde{E}(a)=\;&\pi_{\substack{ {} \\ 0}}\big( E(a)\big)\;\! ,\qquad a\in s(E_0)M^+\! s(E_0)\;\! , \\
E(a)=\;&\pi_0^{-1}\Big( \widetilde{E} \big( s(E_0)\;\! a\;\! s(E_0)\big) \Big)\;\! ,\qquad a\in M^+\;\! .
\end{split}
\end{equation*}
Moreover, by this correspondence
\begin{equation*}
\begin{split}
\mathfrak{M}_{\widetilde{E}}=\;&\mathfrak{M}_{E}\cap\big( s(E_0) M s(E_0)\big) , \\
\mathfrak{M}_{E}\cap M^+=\;&\{ a\in M^+ ; s(E_0)\;\! a\;\! s(E_0)\in \mathfrak{M}_{\widetilde{E}}\}\;\! , \\
s(\widetilde{E})=\;&s(E)\;\! ,\text{in particular, }\widetilde{E_0}\text{ is faithful.}
\end{split}
\end{equation*}
\end{proposition}

\begin{proof}
Let $E : M^+\! \longrightarrow \overline{N}\;\!^+\!$ be a semi-finite, normal operator valued
weight with

\noindent $\;\! s(E)\leq s(E_0)\;\!$.

Since $1_M-p_0\leq 1_M-s(E_0)\leq 1_M-s(E)\;\!$, we have
\medskip

\centerline{$E(1_M-p_0)\leq E \big( 1_M-s(E_0)\big)\leq E \big( 1_M-s(E)\big) =0$}
\smallskip

\noindent and (\ref{reduced-oper.val.}) yields
\begin{equation}\label{range}
E(a)=E( p_0 a p_0) =p_0 E(a) p_0 \in \overline{p_0Np_0}\;\!^+ =\overline{Np_0}\;\!^+\;\! ,\qquad a\in M^+ .
\end{equation}
and
\begin{equation}\label{support}
E(a)=E\big( s(E_0)\;\! a\;\! s(E_0)\big)\;\! ,\qquad a\in M^+ .
\end{equation}

According to (\ref{range}) the composition $\pi_{\substack{ {} \\ 0}}\circ E$ is well defined, so an
operator valued weight $\widetilde{E} : s(E_0)M^+\! s(E_0) \longrightarrow \overline{N s(E_0)}\;\!^+$
can be defined by the formula
\smallskip

\centerline{$\widetilde{E}(a):=\pi_{\substack{ {} \\ 0}}\big( E(a)\big)\;\! ,\qquad a\in s(E_0)M^+\! s(E_0)\;\! .$}

\noindent It is normal, semi-finite and satisfying $s(\widetilde{E})=s(E)\;\!$:
\begin{itemize}
\item The normality is obvious.
\item The semi-finiteness is consequence of the semi-finiteness of $E$ because, by
(\ref{support}), $x\in \mathfrak{N}_{E}\Longleftrightarrow x\;\! s(E_0)\in \mathfrak{N}_{E}\;\!$,
and therefore $\mathfrak{N}_E s(E_0)\subset \mathfrak{N}_{\widetilde{E}}\;\!$.
\item By the definition of $\widetilde{E}\;\!$,
$\widetilde{E}\big( s(E_0)-s(E)\big) =\pi_{\substack{ {} \\ 0}}\Big( E\big( s(E_0)-s(E)\big)\Big) =0\;\!$, so
\smallskip

\centerline{$\qquad\quad s(E_0)-s(E)\leq s(E_0)-s(\widetilde{E})\Longleftrightarrow
s(\widetilde{E})\leq s(E)\;\! .$}
\smallskip

\noindent On the other hand, $0=\widetilde{E}\big( s(E_0)-s(\widetilde{E})\big) =
\pi_{\substack{ {} \\ 0}}\Big( E\big( s(E_0)-s(\widetilde{E})\big)\Big)$ yields
\smallskip

\centerline{$\qquad\quad s(E)-s(\widetilde{E}) s(E)=\big( s(E_0)-s(\widetilde{E})\big) s(E)=0
\Longleftrightarrow s(E)\leq s(\widetilde{E})\;\! .$}
\end{itemize}
\smallskip

For $\mathfrak{M}_{\widetilde{E}}=\mathfrak{M}_{E}\cap\big( s(E_0) M s(E_0)\big)$ it is
enough to show show that
\begin{equation}\label{to-tilde}
\mathfrak{M}_{\widetilde{E}}\cap \big( s(E_0) M^+\! s(E_0)\big) =
\mathfrak{M}_{E}\cap\big( s(E_0) M^+\! s(E_0)\big)
\end{equation}
because $\mathfrak{M}_{E}\cap\big( s(E_0) M s(E_0)\big)$ is the linear span of
$\mathfrak{M}_{E}\cap\big( s(E_0) M^+ s(E_0)\big)$. Indeed, each
$x\in \mathfrak{M}_{E}\cap\big( s(E_0) M s(E_0)\big)$ is on the one hand
linear combination of certain elements $a_j\in \mathfrak{M}_{E}\cap M^+\! ,1\leq j\leq n\;\!$,
and on the other hand $x=s(E_0) x s(E_0)\;\!$, so $x$ is a linear combination of
the elements $s(E_0) a_j s(E_0) \in \mathfrak{M}_{E}\cap\big( s(E_0) M^+\! s(E_0)\big)$.

Now (\ref{to-tilde}) is consequence of the definition of $\widetilde{E}\;\!$. Indeed, for
$a\in s(E_0) M^+\! s(E_0)$ we have :

\centerline{$a\in \mathfrak{M}_{\widetilde{E}}\Longleftrightarrow \widetilde{E}(a)\in N s(E_0)
\Longleftrightarrow E(a)\in N p_0 \overset{(\ref{range})}{\Longleftrightarrow} E(a)\in N .$}
\smallskip

Let next $\widetilde{E} : s(E_0)M^+\! s(E_0) \longrightarrow \overline{N s(E_0)}\;\!^+$
be a semi-finite, normal operator

\noindent valued weight and define the operator valued weight
$E : M^+\! \longrightarrow \overline{N}\;\!^+\!$ by
\medskip

\centerline{$E(a)=\pi_0^{-1}\Big( \widetilde{E} \big( s(E_0)\;\! a\;\! s(E_0)\big) \Big)\;\! ,
\qquad a\in M^+\;\! .$}
\smallskip

\noindent Then $E$ is normal, semi-finite and of support $s(E)\leq s(E_0)\;\!$:
\begin{itemize}
\item The normality is obvious.
\item Since $\big\{ x\big( y+1_M-s(E_0)\big)\;\! ; x\in M ,y\in\mathfrak{N}_{\widetilde{E}}\big\}
\subset \mathfrak{N}_E$ and $s(E_0)$ belongs to the $w$-closure of $\mathfrak{N}_{\widetilde{E}}\;\!$,
it follows that every $x\in M$ belongs to the $w$-closure of $\mathfrak{N}_E\;\!$, that is $E$
is semi-finite.
\item Since $E\;\!$, $E\big( 1_M - s(E_0)\big) =
\pi_0^{-1}\Big( \widetilde{E} \big( s(E_0)\big( 1_M - s(E_0)\big)s(E_0)\big) \Big) =0\;\!$, we

\noindent have $1_M - s(E_0)\leq 1_M - s(E)\Longleftrightarrow s(E)\leq s(E_0)\;\!$.
\end{itemize}

The equality $\mathfrak{M}_{E}\cap M^+=\{ a\in M^+ ; s(E_0)\;\! a\;\! s(E_0)\in
\mathfrak{M}_{\widetilde{E}}\}$ follows by using
\begin{itemize}
\item $\mathfrak{M}_{E}\cap M^+=\{ a\in M^+ ; s(E_0)\;\! a\;\! s(E_0)\in \mathfrak{M}_{E}\}\;\!$,
consequence of (\ref{support}), and
\item the equivalence $s(E_0)\;\! a\;\! s(E_0)\in \mathfrak{M}_{E}\Longleftrightarrow
s(E_0)\;\! a\;\! s(E_0)\in \mathfrak{M}_{\widetilde{E}}$ for $a\in M^+\!$, consequence
of the above proved equality $\mathfrak{M}_{\widetilde{E}}=\mathfrak{M}_{E}\cap
\big( s(E_0) M s(E_0)\big)$.
\end{itemize}

Finally we show that the associations $E\longmapsto \widetilde{E}$ and $\widetilde{E}\longmapsto E$
are mutually inverse applications. Indeed, for each $a\in M^+$,
\smallskip

\centerline{$\pi_0^{-1}\bigg(\! \pi_{\substack{ {} \\ 0}}\Big( E\big( s(E_0)\;\! a\;\! s(E_0)\big) \Big)\! \bigg) =
E\big( s(E_0)\;\! a\;\! s(E_0)\big) \overset{(\ref{support})}{=} E(a)$}
\smallskip

\noindent and, for each $a\in s(E_0)M^+\! s(E_0)\;\!$,
\smallskip

\centerline{$\pi_{\substack{ {} \\ 0}}\bigg(\! \pi_0^{-1}\Big( \widetilde{E}\big( s(E_0)\;\! a\;\! s(E_0)\big)
\Big)\! \bigg) = \widetilde{E}\big( s(E_0)\;\! a\;\! s(E_0)\big) = \widetilde{E}(a)\;\! ,$}
\smallskip

\noindent where the last equality is trivial because $a\in s(E_0)M^+\! s(E_0)\;\!$,

\end{proof}

Now we are ready to describe the situation $E_2 \restriction (\mathfrak{M}_{E_1}\cap M^+) =
E_1 \restriction (\mathfrak{M}_{E_1}\cap M^+)$ for general semi-finite, normal operator valued
weights $E_1\;\! ,E_2 : M\longrightarrow \overline{N}\;\!^+$, $M$ a

\noindent $W^*$-algebra and $1_M\in N\subset M$ a $W^*$-subalgebra.
The proof consists in reduction, based on Proposition \ref{reduce}, to the case of
faithful $E_1\;\!$. We prove also a characterization of this situation in terms of scalar valued
 weights.

\begin{theorem}\label{general-ineq}
Let $M$ be a $W^*$-algebra, $1_M\in N\subset M$ a $W^*$-subalgebra, and

\noindent $E_1\;\! ,E_2 : M^+\longrightarrow \overline{N}\;\!^+\!$ semi-finite, normal
operator valued weights.
\medskip

\hspace{3 pt}$({\rm i})$ If $\mathfrak{M}_{E_1}\subset \mathfrak{M}_{E_2}\text{ and }\;\! E_2(a)=E_1(a)\text{ for
all }a\in \mathfrak{M}_{E_1}\cap M^+$, then $E_2\leq E_1\;\! .$
\medskip

$({\rm ii})$ The condition
\begin{equation}\label{vector}
\mathfrak{M}_{E_1}\! \subset \mathfrak{M}_{E_2}\text{ and }\;\! E_2(a)=E_1(a)\text{ for
all }a\in \mathfrak{M}_{E_1}\cap M^+
\end{equation}
is equivalent to the condition
\begin{equation}\label{scalar}
\begin{split}
&\text{for every faithful, semi-finite, normal weight }\varphi\text{ on }N , \\
&\mathfrak{M}_{\varphi\circ E_1}\! \subset
\mathfrak{M}_{\varphi\circ E_2}\text{ and }\;\! (\varphi\circ E_2)(a)
=(\varphi\circ E_1)(a)\text{ for all }a\in \mathfrak{M}_{\varphi\circ E_1}\cap M^+.
\end{split}
\end{equation}
\end{theorem}

\begin{proof}
The assumptions in both (\ref{vector}) and (\ref{scalar}) imply the inequality $s(E_2)\leq s(E_1)\;\!$.

\noindent Indeed, assuming (\ref{vector}),
\begin{equation*}
\begin{split}
1_M-s(E_1)\in \mathfrak{M}_{E_1} \Longrightarrow\;&
E_2\big( 1_M-s(E_1)\big) =E_1\big( 1_M-s(E_1)\big) =0 \\
\Longrightarrow\;&1_M-s(E_1)\leq 1_M-s(E_2)\Longleftrightarrow s(E_2)\leq s(E_1)\;\! .
\end{split}
\end{equation*}
On the other hand, assuming (\ref{scalar}), we have for any faithful, semi-finite, normal

\noindent weight $\varphi\text{ on }N$:
\begin{equation*}
\begin{split}
1_M-s(E_1)\in \mathfrak{M}_{\varphi\circ E_1} \Longrightarrow\;&
(\varphi\circ E_2)\big( 1_M-s(E_1)\big) =(\varphi\circ E_1)\big( 1_M-s(E_1)\big) =0 \\
\Longrightarrow\;&E_2\big( 1_M-s(E_1)\big) =0 \\
\Longrightarrow\;&1_M-s(E_1)\leq 1_M-s(E_2)\Longleftrightarrow s(E_2)\leq s(E_1)\;\! .
\end{split}
\end{equation*}

Consequently, we can start the proof of both (i) and (ii) with the assumption $s(E_2)\leq s(E_1)\;\!$.
Let $p_1$ denote the central support of $s(E_1)\in N'\cap M$ in $N$.
Using the notations from Proposition \ref{reduce}, we can consider the semi-finite, normal
operator valued weights
$\widetilde{E_j} : s(E_1)M^+\! s(E_1) \longrightarrow \overline{N s(E_1)}\;\!^+\! , j=1\;\! ,2\;\!$,
defined by the formula
\begin{equation}\label{reduced-weight}
\widetilde{E_j}(a)=\pi_1\big( E_j(a)\big)\;\! ,\qquad a\in s(E_1)M^+\! s(E_1)\;\! ,
\end{equation}
where $\pi_1$ denotes the ${}^*$-isomorphism $N p_1\ni y\longmapsto y s(E_1)\in N s(E_1)\;\!$.
According to

\noindent Proposition \ref{reduce}, the operator valued weight $\widetilde{E_1}$ is also
faithful,
\begin{equation}\label{reduced-dom}
\mathfrak{M}_{\widetilde{E_j}}=\mathfrak{M}_{E_j}\cap \big( s(E_1) M s(E_1)\big) ,\qquad
 j=1\;\! ,2\;\! ,
\end{equation}
\begin{equation}\label{extended-dom}
\mathfrak{M}_{E_j}\cap M^+=\{ a\in M^+ ; s(E_1)\;\! a\;\! s(E_1)\in \mathfrak{M}_{\widetilde{E_j}}\}\;\! ,
\quad j=1\;\! ,2\;\! ,
\end{equation}
and the inversion formula
\begin{equation}\label{inversion}
E_j(a)=\pi_1^{-1}\Big( \widetilde{E_j} \big( s(E_1)\;\! a\;\! s(E_1)\big) \Big)\;\! ,\qquad a\in M^+,\;\!
j=1\;\! ,2
\end{equation}
holds true.
\smallskip

Now we pass to the proof of (i).

Assuming (\ref{vector}) and using (\ref{reduced-dom}) and (\ref{reduced-weight}), we deduce :
\medskip

\centerline{$\mathfrak{M}_{\widetilde{E_1}}=\mathfrak{M}_{E_1}\cap \big( s(E_1) M s(E_1)\big)
\subset \mathfrak{M}_{E_2}\cap \big( s(E_1) M s(E_1)\big) =\mathfrak{M}_{\widetilde{E_2}}$}
\smallskip

\noindent and, for each $a\in \mathfrak{M}_{\widetilde{E_1}}\cap \big( s(E_1) M^+\! s(E_1)\big)$,
\medskip

\centerline{$\widetilde{E_2}(a)=\pi_1\big( E_2(a)\big) =\pi_1\big( E_1(a)\big) =\widetilde{E_1}(a)\;\! .$}
\smallskip

\noindent Since $\widetilde{E_1}$ is faithful, Lemma \ref{faithful-ineq} entails that
$\widetilde{E_2}\leq \widetilde{E_1}$ and using (\ref{inversion}), we conclude that $E_2\leq E_1\;\!$.
\smallskip

Concerning (ii), it will be implied by Lemmas \ref{faithful-ineq} and  \ref{inv-faithful-ineq} once we prove that

\noindent (\ref{vector}) is equivalent to
\begin{equation}\label{vector-red}
\mathfrak{M}_{\widetilde{E_1}}\! \subset \mathfrak{M}_{\widetilde{E_2}}\text{ and }\;\! {\widetilde{E_2}}(a)=
{\widetilde{E}}_1(a)\text{ for all }a\in \mathfrak{M}_{\widetilde{E_1}}\cap \big( s(E_1) M^+\! s(E_1)\big)
\tag{\ref{vector}-bis}
\end{equation}
and (\ref{scalar}) is equivalent to
\begin{equation}\label{scalar-red}
\begin{split}
&\text{for every faithful, semi-finite, normal weight }\psi\text{ on }Ns(E_1)\;\! , \\
&\mathfrak{M}_{\psi\circ {\widetilde{E_1}}}\! \subset
\mathfrak{M}_{\psi\circ {\widetilde{E_2}}}\text{ and }\;\! (\psi\circ {\widetilde{E_2}})(a)
=(\psi\circ {\widetilde{E_1}})(a)\text{ holds true} \\
&\text{for all }a\in \mathfrak{M}_{\psi\circ {\widetilde{E_1}}}\cap \big( s(E_1) M^+\! s(E_1)\big)\;\! .
\end{split}
\tag{\ref{scalar}-bis}
\end{equation}

Implication (\ref{vector})$\;\!\Longrightarrow\;\!$(\ref{vector-red}) was already shown
and used in the proof of (i), it is immediate consequence of
(\ref{reduced-dom}) and (\ref{reduced-weight}). Also the proof of the inverse implication
(\ref{vector-red})$\;\!\Longrightarrow\;\!$(\ref{vector}) is easy, it follows immediately by
using (\ref{extended-dom}) and (\ref{inversion}).
\smallskip

Let us next prove (\ref{scalar})$\;\!\Longrightarrow\;\!$(\ref{scalar-red}).

For let $\psi$ be any faithful, semi-finite, normal weight on $Ns(E_1)\;\!$.
Choosing some faithful, semi-finite, normal weight $\theta$ on $N (1_M-p_1)\;\!$, let us consider
the faithful, semi-finite, normal weight $\varphi :=\psi\circ \pi_1 +\theta$ on $N$.
By (\ref{scalar}) we have
\begin{equation}\label{by-scalar}
\mathfrak{M}_{\varphi\circ E_1}\! \subset
\mathfrak{M}_{\varphi\circ E_2}\text{ and }\;\! (\varphi\circ E_2)(a)
=(\varphi\circ E_1)(a)\text{ for all }a\in \mathfrak{M}_{\varphi\circ E_1}\cap M^+.
\end{equation}

Now let $a\in \mathfrak{M}_{\psi\circ {\widetilde{E_1}}}\cap \big( s(E_1) M^+\! s(E_1)\big)$
be arbitrary. Since
\smallskip

\centerline{$a\in s(E_1) M^+\! s(E_1)\subset p_1M^+p_1\Longrightarrow
E_j(a)\in p_1\overline{N}\;\!^+p_1\Longrightarrow \theta\big( E_j(a)\big) =0\;\! ,\quad
j=1\;\! ,2\;\! ,$}
\smallskip

\noindent we have
\smallskip

\centerline{$(\varphi\circ E_1)(a)=\psi\Big( \pi_1\big( E_1(a)\big)\Big) +\theta\big( E_1(a)\big)
\overset{(\ref{reduced-weight})}{=}(\psi\circ \widetilde{E_1})(a)<+\infty\;\! .$}
\smallskip

\noindent Thus $a\in \mathfrak{M}_{\varphi\circ E_1}\! \subset \mathfrak{M}_{\varphi\circ E_2}$
and, using (\ref{by-scalar}), we conclude :
\begin{equation*}
\begin{split}
(\psi\circ \widetilde{E_1})(a)=\;&(\varphi\circ E_1)(a)=(\varphi\circ E_2)(a)=
\psi\Big( \pi_1\big( E_2(a)\big)\Big) +\theta\big( E_2(a)\big) \\
\overset{(\ref{reduced-weight})}{=}&(\psi\circ \widetilde{E_2})(a)\;\! .
\end{split}
\end{equation*}

Finally we prove also the inverse implication (\ref{scalar-red})$\;\!\Longrightarrow\;\!$(\ref{scalar}).

For let $\varphi$ be any faithful, semi-finite, normal weight on $N$.
We define the faithful, semi-finite, normal weight $\psi$ on $N s(E_1)$ by the formula
\smallskip

\centerline{$\psi (b):=\varphi \big( \pi_1^{-1}(b)\big)\;\! ,\qquad b\in \big( N s(E_1)\big)^+.$}
\smallskip

\noindent By (\ref{scalar-red}) we have
\begin{equation}\label{by-scalar-red}
\begin{split}
&\mathfrak{M}_{\psi\circ {\widetilde{E_1}}}\! \subset
\mathfrak{M}_{\psi\circ {\widetilde{E_2}}}\text{ and }\;\! (\psi\circ {\widetilde{E_2}})(a)
=(\psi\circ {\widetilde{E_1}})(a)\text{ holds true} \\
&\text{for all }a\in \mathfrak{M}_{\psi\circ {\widetilde{E_1}}}\cap \big( s(E_1) M^+\! s(E_1)\big)\;\! .
\end{split}
\end{equation}

Let $a\in \mathfrak{M}_{\varphi\circ E_1}\cap M^+$ be arbitrary. Then
\begin{equation*}
+\infty >(\varphi\circ E_1)(a) \overset{(\ref{inversion})}{=} \varphi \bigg(\! 
\pi_1^{-1}\Big( \widetilde{E_j} \big( s(E_1)\;\! a\;\! s(E_1)\big) \Big)\! \bigg)
=(\psi\circ \widetilde{E_1}) \big( s(E_1)\;\! a\;\! s(E_1)\big) .
\end{equation*}
Thus $s(E_1)\;\! a\;\! s(E_1)\in \mathfrak{M}_{\psi\circ {\widetilde{E_1}}}\! \subset
\mathfrak{M}_{\psi\circ {\widetilde{E_2}}}$ and, using (\ref{by-scalar-red}), we conclude :
\begin{equation*}
\begin{split}
(\varphi\circ E_1)(a) =\;\;& (\psi\circ \widetilde{E_1}) \big( s(E_1)\;\! a\;\! s(E_1)\big)
=(\psi\circ \widetilde{E_2}) \big( s(E_1)\;\! a\;\! s(E_1)\big) \\
\overset{(\ref{inversion})}{=}&\psi \Big( \pi_1 \big( E_2(a) \big)\! \Big)
=(\varphi\circ E_2)(a)\;\! .
\end{split}
\end{equation*}

\end{proof}

The next equality criterion is an immediate consequence of Theorem \ref{general-ineq} :

\begin{corollary}\label{general-eq}
Let $M$ be a $W^*$-algebra, and $1_M\in N\subset M$ a $W^*$-subalgebra. If

\noindent $E_1\;\! ,E_2 : M^+\longrightarrow \overline{N}\;\!^+$ are semi-finite, normal
operator valued weights such that
\medskip

\centerline{$\mathfrak{M}_{E_1}=\mathfrak{M}_{E_2}\text{ and }\;\! E_2(a)=E_1(a)\text{ for
all }a\in \mathfrak{M}_{E_1}\cap M^+=\mathfrak{M}_{E_2}\cap M^+\;\! ,$}
\medskip

\noindent then $E_2=E_1\;\!$.
\end{corollary}

\hfill $\square$
\medskip

For $E_1\;\! ,E_2$ of equal supports, the situation $E_2 \restriction (\mathfrak{M}_{E_1}\cap M^+) =
E_1 \restriction (\mathfrak{M}_{E_1}\cap M^+)$ is equivalent also to a weakened version of
(\ref{scalar}). Let us first consider the case of

\noindent faithful operator valued weights and formulate the
following variant of \cite{H2},  Remark on page 360 :

\begin{lemma}\label{cocycle-cond}
Let $M$ be a $W^*$-algebra, $1_M\in N\subset M$ a $W^*$-subalgebra, and $E_1\;\! ,E_2$
two faithful, semi-finite, normal operator valued weights $M^+\! \longrightarrow \overline{N}\;\!^+$.
If for some faithful, semi-finite, normal weight $\varphi$ on $N$ we have
\begin{equation}\label{scalar-cond}
\mathfrak{M}_{\varphi\circ E_1}\! \subset
\mathfrak{M}_{\varphi\circ E_2}\text{ and }\;\! (\varphi\circ E_2)(a)
=(\varphi\circ E_1)(a)\text{ for all }a\in \mathfrak{M}_{\varphi\circ E_1}\cap M^+,
\end{equation}
then {\rm (\ref{scalar-cond})} holds true for every faithful, semi-finite, normal weight $\varphi$ on $N$.
\end{lemma}

\begin{proof}
We have just to follow the reasoning in \cite{H2},  Remark on page 360, using \cite{Co2}, Th\'eor\`eme 3
instead of \cite{Co2}, Th\'eor\`eme 3 (that is the above Theorem \ref{cocycle} (iii) instead of Theorem
\ref{cocycle} (ii) ). Let us sketch the details.
\smallskip

Assuming that (\ref{scalar-cond}) holds true for some faithful, semi-finite, normal weight $\varphi_1$
on $N$, let $\varphi_2$ be any other faithful, semi-finite, normal weight on $N$.
According to Theorem \ref{cocycle} (iii), $\mathbb{R}\ni t\longmapsto
\big( D(\varphi_1\circ E_2) : D(\varphi_1\circ E_1)\big)_t\in M$ has a $w$-continuous extension
$\{\zeta\in\mathbb{C}\;\! ; -\frac 12\leq {\rm Im} \zeta\leq 1\}\ni \zeta\longmapsto
\big( D(\varphi_1\circ E_2) : D(\varphi_1\circ E_1)\big)_{\zeta}\in M\;\!$, which is analytic in the interior
and such that $\big( D(\varphi_1\circ E_2) : D(\varphi_1\circ E_1)\big)_{-\frac i2}$ is isometric.
On the other hand, by \cite{H2}, Proposition 6.1 (2), we have
\smallskip

\centerline{$\big( D(\varphi_2\circ E_2) : D(\varphi_2\circ E_1)\big)_t =
\big( D(\varphi_1\circ E_2) : D(\varphi_1\circ E_1)\big)_t$ for all $t\in\mathbb{R}\;\! .$}
\smallskip

\noindent Thus $\mathbb{R}\ni t\longmapsto
\big( D(\varphi_2\circ E_2) : D(\varphi_2\circ E_1)\big)_t\in M$ allows the extension
\medskip

\centerline{$\{\zeta\in\mathbb{C}\;\! ; -\frac 12\leq {\rm Im} \zeta\leq 1\}\ni \zeta\longmapsto
\big( D(\varphi_2\circ E_2) : D(\varphi_2\circ E_1)\big)_{\zeta}$}

\noindent\hspace{6.35 cm}$=\big( D(\varphi_1\circ E_2) : D(\varphi_1\circ E_1)\big)_{\zeta}\;\! ,$
\smallskip

\noindent which is $w$-continuous, analytic in the interior, and such that
\medskip

\centerline{$\big( D(\varphi_2\circ E_2) : D(\varphi_2\circ E_1)\big)_{- \frac i2}=
\big( D(\varphi_1\circ E_2) : D(\varphi_1\circ E_1)\big)_{- \frac i2}$ is isometric.}
\smallskip

\noindent Using now again Theorem \ref{cocycle} (iii), we infer that (\ref{scalar-cond}) holds
true for $\varphi_2\;\!$.

\end{proof}

Now we complete, in the case $E_1\;\! ,E_2$ have equal suports, the characterization of the situation
$E_2 \restriction (\mathfrak{M}_{E_1}\cap M^+) =E_1 \restriction (\mathfrak{M}_{E_1}\cap M^+)$
given in Theorem \ref{general-ineq} :

\begin{theorem}\label{compl-ineq}
Let $M$ be a $W^*$-algebra, $1_M\in N\subset M$ a $W^*$-subalgebra, and

\noindent $E_1\;\! ,E_2 : M^+\longrightarrow \overline{N}\;\!^+\!$ semi-finite, normal
operator valued weights of equal supports.
Then the following conditions are equivalent $:$
\medskip

\hspace{7 pt}$(i)$ $\mathfrak{M}_{E_1}\subset \mathfrak{M}_{E_2}\text{ and }\;\! E_2(a)=E_1(a)\;\! ,
a\in \mathfrak{M}_{E_1}\cap M^+$

\noindent\hspace{1.16 cm}$($implying, by {\rm Theorem \ref{general-ineq} (i)}, $E_2\le E_1)$.
\medskip

\hspace{3.5 pt}$(ii)$ $\mathfrak{M}_{\varphi\circ E_1}\! \subset
\mathfrak{M}_{\varphi\circ E_2}\text{ and }\;\! (\varphi\circ E_2)(a)
=(\varphi\circ E_1)(a)\;\! ,a\in \mathfrak{M}_{\varphi\circ E_1}\cap M^+$

\noindent\hspace{1.16 cm}for every faithful, semi-finite, normal weight $\varphi$ on $N$.
\medskip

$(iii)$ $\mathfrak{M}_{\varphi\circ E_1}\! \subset
\mathfrak{M}_{\varphi\circ E_2}\text{ and }\;\! (\varphi\circ E_2)(a)
=(\varphi\circ E_1)(a)\;\! ,a\in \mathfrak{M}_{\varphi\circ E_1}\cap M^+$

\noindent\hspace{1.16 cm}for some faithful, semi-finite, normal weight $\varphi$ on $N$.
\end{theorem}

\begin{proof}
Equivalence (i)$\Longleftrightarrow$(ii) was proved in Theorem \ref{general-ineq} (even
without assuming the equality of the supports of $E_1$ and $E_2$), while implication
(ii)$\Longrightarrow$(iii) is trivial. Thus the proof will be done once we prove
(iii)$\Longrightarrow$(ii), what will be performed by reduction to Lemma \ref{cocycle-cond}.
\smallskip

Let $p$ denote the central support of $s:=s(E_1)=s(E_2)\in N'\cap M$ in $N$.
Using the notations from Proposition \ref{reduce}, we can consider the semi-finite, normal
operator valued weights
$\widetilde{E_j} : s M^+\! s \longrightarrow \overline{(N s)}\;\!^+\! , j=1\;\! ,2\;\!$,
defined by the formula
\begin{equation}\label{reduced-weight-again}
\widetilde{E_j}(a)=\pi\big( E_j(a)\big)\;\! ,\qquad a\in s M^+\! s\;\! ,
\end{equation}
where $\pi$ denotes the ${}^*$-isomorphism $N p\ni y\longmapsto y\;\! s\in N s\;\!$.
According to Proposition \ref{reduce}, the operator valued weights $\widetilde{E_1}$ and
$\widetilde{E_2}$ are also faithful.
\smallskip

Now let us assume that
\begin{equation}\label{by-scalar-in}
\mathfrak{M}_{\varphi_1\circ E_1}\! \subset
\mathfrak{M}_{\varphi_1\circ E_2}\text{ and }\;\! (\varphi_1\circ E_2)(a)
=(\varphi_1\circ E_1)(a)\;\! ,a\in \mathfrak{M}_{\varphi_1\circ E_1}\cap M^+
\end{equation}
holds true for some faithful, semi-finite, normal weight $\varphi_1$ on $N$.
We have to show that, for any faithful, semi-finite, normal weight $\varphi_2$
on $N$,
\begin{equation}\label{by-scalar-out}
\mathfrak{M}_{\varphi_2\circ E_1}\! \subset
\mathfrak{M}_{\varphi_2\circ E_2}\text{ and }\;\! (\varphi_2\circ E_2)(a)
=(\varphi_2\circ E_1)(a)\;\! ,a\in \mathfrak{M}_{\varphi_2\circ E_1}\cap M^+
\end{equation}
still holds true.

We define the faithful, semi-finite, normal weights $\psi_k\;\! ,k=1\;\! ,2\;\!$, on $N s$
by the formula

\centerline{$\psi_k (b):=\varphi_k\big( \pi^{-1}(b)\big)\;\! ,\qquad b\in (N s)^+.$}
\medskip

\noindent Then, for each $j=1\;\! ,2$ and $k=1\;\! ,2\;\!$, we have
\begin{equation}\label{value}
(\psi_k\circ \widetilde{E_j})(a)\overset{(\ref{reduced-weight-again})}{=}
\psi_k\Big( \pi\big( E_j(a)\big)\Big) =\varphi_k \big( E_j(a)\big)
\;\! ,\qquad a\in s M^+\! s\;\! ,
\end{equation}
in particular
\begin{equation}\label{domain}
\mathfrak{M}_{\psi_k\circ \widetilde{E_j}} \cap (s M^+s)
= \mathfrak{M}_{\varphi_k\circ E_j} \cap (s M^+s)\;\! .
\end{equation}

Now, by (\ref{domain}) and (\ref{by-scalar-in}) we have
\medskip

\centerline{$\mathfrak{M}_{\psi_1\circ \widetilde{E_1}}\! \cap (s M^+s) =
\mathfrak{M}_{\varphi_1\circ E_1}\! \cap (s M^+s) \subset
\mathfrak{M}_{\varphi_1\circ E_2}\! \cap (s M^+s) =
\mathfrak{M}_{\psi_1\circ \widetilde{E_2}}\! \cap (s M^+s)\;\! ,$}
\medskip

\noindent and by (\ref{value}) and (\ref{by-scalar-in}) we obtain for every
$a\in \mathfrak{M}_{\psi_1\circ \widetilde{E_1}} \cap (s M^+s)\;\!$:
\smallskip

\centerline{$(\psi_1\circ \widetilde{E_2})(a)=\varphi_1 \big( E_2(a)\big) =
\varphi_1 \big( E_1(a)\big) =(\psi_1\circ \widetilde{E_1})(a)\;\! .$}
\medskip

\noindent Thus we can apply Lemma \ref{cocycle-cond} deducing
\begin{equation}\label{by-scalar-out-red}
\mathfrak{M}_{\psi_2\circ \widetilde{E_1}}\! \subset
\mathfrak{M}_{\psi_2\circ \widetilde{E_2}}\text{ and }\;\! (\psi_2\circ \widetilde{E_2})(a)
=(\psi_2\circ \widetilde{E_1})(a)\;\! ,a\in \mathfrak{M}_{\psi_2\circ \widetilde{E_1}}\cap
(s M^+s)\;\! .
\end{equation}

Let $a\in \mathfrak{M}_{\varphi_2\circ E_1}\cap M^+$ be arbitrary. By (\ref{value}) and
(\ref{reduced})
we have
\smallskip

\centerline{$(\psi_k\circ \widetilde{E_j})(s\;\! a\;\! s)\overset{(\ref{value})}{=}
(\varphi_2\circ E_1)(s\;\! a\;\! s)\overset{(\ref{reduced})}{=}(\varphi_2\circ E_1)(a)<+\infty\;\! ,$}
\medskip

\noindent so $s\;\! a\;\! s\in \mathfrak{M}_{\psi_2\circ \widetilde{E_1}}\cap (s M^+s)\;\! .$
Using (\ref{reduced}), (\ref{value}) and (\ref{by-scalar-out-red}), we obtain
\smallskip

\centerline{$(\varphi_2\circ E_2)(a)\overset{(\ref{reduced})}{=}(\varphi_2\circ E_2)(s\;\! a\;\! s)
\overset{(\ref{value})}{=}(\psi_2\circ \widetilde{E_2})(s\;\! a\;\! s)
\overset{(\ref{by-scalar-out-red})}{=}(\psi_2\circ \widetilde{E_1})(s\;\! a\;\! s)$}
\smallskip

\noindent\hspace{2.55 cm}$\overset{(\ref{value})}{=}\! (\varphi_2\circ E_1)(s\;\! a\;\! s)
\overset{(\ref{reduced})}{=}(\varphi_2\circ E_1)(a)<+\infty\;\! .$
\smallskip

\noindent In other words $a\in \mathfrak{M}_{\varphi_2\circ E_2}\cap M^+$ and
$(\varphi_2\circ E_2)(a)=(\varphi_2\circ E_1)(a)\;\!$. This proves

\noindent (\ref{by-scalar-out}).

\end{proof}

\begin{remark}
It is possible that the statement of Theorem \ref{compl-ineq} holds without the assumption
$s(E_1)=s(E_2)\;\!$. This would follow if Lemma \ref{cocycle-cond} would hold without
assuming the faithfulness of $E_2\;\!$.

Actually the cocycle $(D\psi : D\varphi )_t\;\! ,t\in \mathbb{R}\;\!$, of a not necessarily faithful,
semi-finite, normal weight $\psi$ on a $W^*$-algebra $M$ with respect to a faithful, semi-finite,
normal weight $\varphi$ on $M$ was already considered by A. Connes and M. Takesaki
in \cite{CoT}, I.1, pages 478-479 (see also \cite{St}, Theorem 3.1).
If
\begin{itemize}
\item Theorem \ref{cocycle} would keep validity even in this setting and
\item the proof of \cite{H2}, Proposition 6.1 could be adapted to the situation of semi-finite, normal
operator valued weights $E_1\;\! ,E_2$ from $M$ to a $W^*$-subalgebra $1_M\in N\subset M$,
where only $E_1$ assumed to be faithful, and thus prove that the cocycle
$\big( D(\theta\circ E_2) : D(\theta\circ E_1) \big)_t\;\! ,t\in \mathbb{R}\;\!$, does
not depend on the faithful, semi-finite, normal weight $\theta$ on $N$,
\end{itemize}
then the proof of Lemma \ref{cocycle-cond} would work without assuming the faithfulness of $E_2\;\!$.
\end{remark}

\bibliographystyle{amsplain}

\end{document}